\documentclass[11pt, reqno]{amsart}

\usepackage{graphicx,color,amsmath,amssymb,mathrsfs}
\usepackage{amsfonts,amsthm,epsfig,epstopdf}

\newtheorem{theorem}{Theorem}[section]
\newtheorem{lemma}[theorem]{Lemma}
\newtheorem{proposition}[theorem]{Proposition}
\newtheorem{corollary}[theorem]{Corollary}
\newtheorem{remark}[theorem]{Remark}

\newcommand{\Rn}{{\mathbb{R}^n}}
\newcommand{\R}{{\mathbb{R}}}
\newcommand{\e}{{\epsilon}}
\newcommand{\g}{{\gamma}}

\newcommand{\ka}{{\kappa}}
\newcommand{\pa}{{\partial}}
\newcommand{\C}{{\bf{C}}}
\newcommand{\M}{{\mathcal{M}}}
\newcommand{\cc}{{\bf{c}}}
\newcommand{\dd}{{\rm{d}}}
\newcommand{\SSS}{{S_{p,n}}}
\newcommand{\SSp}{{S^p_{p,n}}}
\newcommand{\SSpm}{{S^{-p}_{p,n}}}

\newcommand{\vertiii}[1]{{\left\vert\kern-0.25ex\left\vert\kern-0.25ex\left\vert #1 
    \right\vert\kern-0.25ex\right\vert\kern-0.25ex\right\vert}}

\newcommand{\vphi}{{\varphi}}
\newcommand{\veee}{{c_0v_{\lambda_0, y_0}}}

\newcommand{\vone}{{v_{\lambda_0, y_0}}}

\newcommand{\LL}{{\mathcal{L}}}
\newcommand{\DIVV}{{\rm div \,}}
\newcommand{\spann}{{\rm span \,}}
\newcommand{\pl}{\partial_{\lambda}}
\newcommand{\bs}{\backslash}
\newcommand{\Om}{\Omega}
\newcommand{\om}{\omega}
\newcommand{\na}{{\nabla}}

\numberwithin{equation}{section}
\begin{document}

\title[Gradient stability for the Sobolev inequality: the case $p\geq 2$]{Gradient stability for the Sobolev inequality:\\ the case $p\geq 2$}

\author[Figalli]{Alessio Figalli}
\address[Alessio Figalli]{\newline
 \indent  Department of Mathematics, ETH Z\"{u}rich, \newline 
\indent  HG G 63.2, R\"{a}mistrasse 101, CH-8092 Z\"{u}rich,
Switzerland}
\email{alessio.figalli@math.ethz.ch}

\author[Neumayer]{Robin Neumayer}
\address[Robin Neumayer]{ \newline
\indent  Department of Mathematics, The University of Texas at Austin,  \newline
\indent 2515 Speedway Stop C1200, Austin, TX 78712, USA}
\email{rneumayer@math.utexas.edu}

\begin{abstract}
We prove a strong form of the quantitative Sobolev inequality in $\mathbb{R}^n$ for $p\geq 2$, where the deficit of a function $u\in \dot W^{1,p} $ controls $\| \nabla u -\nabla v\|_{L^p}$ for an extremal function $v$ in the Sobolev inequality.
  \end{abstract}

\maketitle

\section{Introduction}
Given $n\geq 2$ and $1\leq p<n$, the Sobolev inequality provides a control of the $L^r$ norm of a function in terms of a suitable $L^p$ norm of its gradient.
 More precisely, setting $p^*:=np/(n-p)$,
one defines the homogeneous Sobolev space $\dot W^{1,p}$ as the space of functions in $\R^n$ such that $u \in L^{p^*}$ and $|\na u|\in L^p$.
Then the following holds:
 \begin{equation}\label{Sobolev} \| \na u \|_{L^p} \geq S_{p,n} \|u\|_{L^{p^*}}\qquad \forall\,u \in \dot W^{1,p}. \end{equation}
Throughout the paper, all the integrals and function spaces will be over $\Rn$, so we will omit the domain of integration when no confusion arises.

It is well known that the optimal constant in  \eqref{Sobolev} is given by
$$\SSS = \sqrt{\pi} n^{1/p}\left(\frac{n-p}{p-1}\right)^{(p-1)/p} \left(\frac{\Gamma(n/p)\Gamma(1+n - n/p)}{\Gamma(1+ n/2)\Gamma(n)}\right)^{1/n},$$
and that equality is attained in \eqref{Sobolev}  if and only if $u$ belongs to the family of functions
$$cv_{\lambda, y}(x) = c\lambda^{n/p^*} v_1(\lambda(x-y)),\qquad c \in \R,\,\lambda\in\R_+,\,y\in\R^n,$$
where
\begin{equation}\label{v1}
v_1(x) := \frac{\kappa_0}{(1+|x|^{p'})^{(n-p)/p}},
\end{equation}
 see \cite{talenti1976best, aubin1976} and \cite{cordero2004mass} (here $\kappa_0$ is chosen so that $\|v_1\|_{L^{p^*}}=1$, therefore $\|cv_{\lambda, y}\|_{L^{p^*}}=c$, and $p' := p/(p-1)$ denotes the H\"{o}lder conjugate of $p$). 
In other words, 
 \begin{equation}\label{M}
 \M := \left\{ cv_{\lambda, y} : c\in \R,\, \lambda \in \R_+,\, y \in \Rn \right\}
 \end{equation}
is the $(n+2)$-dimensional manifold of extremal functions in the Sobolev inequality \eqref{Sobolev}. 

To quantify how close a function $u\in \dot W^{1,p}$ is to achieving equality in \eqref{Sobolev},
we define its \textit{deficit} to be the $p$-homogeneous functional
$$\delta(u) := \| \na u\|_{L^p}^p - \SSp\|u\|_{L^{p^*}}^p .$$
By \eqref{Sobolev}, the deficit is nonnegative and equals zero if and only if $v\in \M$.
In \cite{brezis1985sobolev}, Brezis and Lieb raised the question of \textit{stability} for the Sobolev inequality, that is, whether the deficit controls an appropriate distance between a function $u\in \dot W^{1,p}$ and the family of extremal functions.

This question was first answered in the case $p=2$ by Bianchi and Egnell in \cite{BianchiEgnell91}: there, they showed that the deficit of a function $u$ controls the $L^2$ distance between the gradient of $u$ and the gradient of  closest extremal function $v$. The result is optimal both in the strength of the distance and the exponent of decay. However, their proof is very specific to the case $p=2$, as it strongly exploits the Hilbert structure of $\dot W^{1,2}$.
Later on, in \cite{ciafusmag07}, Cianchi, Fusco, Maggi, and Pratelli considered the case $1<p<n$ and provided a stability result in which the deficit controls the $L^{p^*}$ distance between $u$ and some $v\in \M$. Their proof uses a combination of symmetrization techniques and tools from the theory of mass transportation. 
More recently, in \cite{figmagpraa},  Figalli, Maggi, and Pratelli used rearrangement techniques and mass transportation theory to show that, in the case $p=1$,
the deficit controls the appropriate notion of distance of $u$ from $\M$ at the level of gradients (see also \cite{FMPSobolev, cianchi06} for partial results when $p=1$). As in \cite{BianchiEgnell91}, the distance considered in \cite{figmagpraa} is the strongest that one expects to control and the exponent of decay is sharp.

In view of \cite{BianchiEgnell91} and \cite{figmagpraa}, one may expect that, for all $1<p<n$, the deficit controls the $L^p$ distance between $\nabla u$ and $\nabla v$ for some $v \in \M$; this would answer the question of Breizis and Lieb in the affirmative with the deficit controlling the strongest possible notion of distance in this setting.
The main result of this paper shows that, in the case $p\geq 2$, this result is indeed true. More precisely, our main result states the following:

\begin{theorem}\label{MainThm} Let $2 \leq p<n$.
There exists a constant $C>0$, depending only on $p$ and $n$, such that for all $u \in \dot W^{1,p},$
\begin{equation} \label{fullstability}
 \| \na u - \na v\|_{L^p}^p \leq C\, \delta(u) +C \|u\|_{L^{p^*}}^{p-1} \|u-v\|_{L^{p^*}}
 \end{equation}
for some $v\in \M$.
\end{theorem}

As a consequence of Theorem~\ref{MainThm} and the main result of \cite{ciafusmag07} (see Theorem~\ref{CFMP} below), we deduce the following
corollary, proving the desired stability at the level of gradients:
\begin{corollary}\label{MainCor} Let $2 \leq p<n$.
There exists a constant $C>0$, depending only on $p$ and $n$, such that for all $u \in \dot W^{1,p},$ \begin{equation} \label{fullstability2}
 \left(\frac{ \| \na u - \na v\|_{L^p}}{\|\na u\|_{L^{p}}}\right)^{\zeta}  \leq C\frac{\delta(u)}{ \| \na u\|_{L^p}^p}
 \end{equation}
 for some $v\in \M$, where $\zeta = p^*p\left(3 + 4p -\frac{3p+1}{n}\right)^2.$
\end{corollary}

The topic of stability for functional and geometric inequalities has generated much interest in recent years. In addition to the aforementioned papers, results of this type have been addressed for the isoperimetric inequality \cite{FMP08,FiMP10,CiLe12}, log-Sobolev inequality \cite{IndreiMarcon, Bobkovetal, fathi2014quantitative},
the higher order Sobolev inequality \cite{GazzolaWeth, BaWeWi}, the fractional Sobolev inequality \cite{ChenFrankWeth}, the Morrey-Sobolev inequality \cite{cianchi08} and the Gagliardo-Nirenberg-Sobolev inequality \cite{CarlenFigalli, Ruffini14}, as well as for numerous other geometric inequalities. Aside from their intrinsic interest, stability results have applications
 in the study of geometric problems (see \cite{figallimaggi11, FigalliMaggi13, CicaleseSparado13})
and can be used to obtain quantitative rates of convergence for diffusion equations
(as in \cite{CarlenFigalli}).

For the remainder of the paper, we will always assume that $2 \leq p <n.$
\\

\noindent{\it Acknowledgments:} A. Figalli is partially supported by NSF Grants DMS-1262411 and DMS-1361122.
R. Neumayer is supported by the NSF Graduate Research Fellowship under Grant DGE-1110007. Both authors warmly thank Francesco Maggi for useful discussions regarding this work.

\section{Theorem \ref{MainThm}: idea of the proof}
As a starting point to prove stability of \eqref{Sobolev} at the level of gradients, one would like to follow the argument used to prove the analogous result in \cite{BianchiEgnell91}. However, this approach turns out to be sufficient only in certain cases, and additional  ideas are needed to conclude the proof. Indeed, a Taylor expansion of the deficit $\delta(u)$ and a spectral gap for the linearized problem allow us to show that the second variation is strictly positive, but in general we cannot absorb the higher order terms. Let us provide a few more details to see to what extent this approach works, where it breaks down, and how we get around it.
\subsection{The expansion approach}

The first idea of the proof of Theorem~\ref{MainThm} is in the spirit of the stability result of Bianchi and Egnell in \cite{BianchiEgnell91}. Ultimately, this approach will need modification, but let us sketch how such an argument would go. 

In order to introduce a Hilbert space structure to our problem, 
we define a weighted $L^2$-type distance of a function $u \in \dot W^{1,p}$ to $\M$ at the level of gradients.
To this end, for each $v=cv_{\lambda,y}\in \M$, we define 
\begin{equation}\label{Matrix A}
A_{v} (x) :=  (p-2) |\na v|^{p-2} \hat{r} \otimes \hat{r} + |\na v|^{p-2} \text{Id},\qquad\hat{r} = \frac{x-y}{ |x -y|},
\end{equation}
 where $(a \otimes b)c := (a\cdot c)b$. Then, with the notation $A_v[a,a] := a^T A_va$ for $a \in \mathbb{R}^n$, we define the weighted $L^2$ distance of $u$ to $\M$ by
\begin{equation}\label{distance}\begin{split}
\dd(u, &\M) := \inf \left\{\Big( \int A_v[\na u -\na v, \na u - \na v] \Big)^{1/2} : \ v \in \M, \ \| v \|_{L^{p^*}} = \|u \|_{L^{p^*}} \right\}\\
&=\inf \left\{\Big( \int A_{cv_{\lambda,y}}[\na u-\na c v_{\lambda, y},\na u-\na c v_{\lambda, y}]
 \Big)^{1/2}  :  \ \lambda \in \R_+,\  y \in \Rn, \ c =\|u \|_{L^{p^*}}  \right\}.
 \end{split}
\end{equation}
Note that
$$ \int A_v[\na u -\na v, \na u - \na v] = 
 \int|\na v|^{p-2} |\na u - \na v|^2
+(p-2) \int|\na v|^{p-2} |\pa_r u - \pa_r v|^2.$$
A few remarks about this definition are in order.
\begin{remark}
\label{rmk:dist}{\rm{The motivation to define $\dd(u,\M)$ in this way instead of, for instance, 
$$\inf \left\{\Big( \int|\na v|^{p-2} |\na u - \na v|^2
 \Big)^{1/2} : \ v \in \M, \ \| v \|_{L^{p^*}} = \|u \|_{L^{p^*}} \right\},$$ will become apparent in Section~\ref{Preliminaries}. 
 This choice, however, is only technical, as
 $$ \int|\na v|^{p-2} |\na u - \na v|^2 \leq \int A_v[\na u  - \na v, \na u - \na v] \leq (p-1)  \int|\na v|^{p-2} |\na u - \na v|^2.$$
 }}
 \end{remark}
\begin{remark}{\rm{One could alternatively define the distance in \eqref{distance} without the constraint $c =\|u \|_{L^{p^*}},$ instead also taking the infimum over the parameter $c$. Up to adding a small positivity constraint to ensure that the infimum is not attained at $v=0$, this definition works, but ultimately the current presentation is more straightforward.}}
\end{remark}
\begin{remark}{\rm{The distance $\dd(u,\M)$ has homogeneity $p/2$, that is, 
$\dd(cu, \M) = c^{p/2} \dd(u,\M).$}}
\end{remark}
 
 In Proposition~\ref{small delta}(1), we show that there exists $\delta_0=\delta_0(n,p)>0$ such that if 
\begin{equation}\label{assumption2} \delta(u) \leq \delta_0 \| \na u \|_{L^p}^p,
\end{equation}
then the infimum in $\dd(u,\M)$ is attained.
Given a function $u\in \dot W^{1,p}$ satisfying \eqref{assumption2}, let $v\in \M$ attain the infimum in \eqref{distance} and define
$$
\varphi:=\frac{u-v}{\|\na (u-v)\|_{L^p}},
$$
so that  $u = v + \e \vphi$ with $\e=\|\na (u-v)\|_{L^p}$ and $\int |\na \vphi |^p  = 1$.
Since $\delta \geq 0$ and $\delta (v) = 0$, the 
Taylor expansion of the deficit of $u$ around $v$ vanishes both at the zeroth and first order. Thus, the expansion leaves us with 
\begin{equation}\label{star}\delta(u) =\e^2 p \int A_v[\na \vphi, \na \vphi]  - \e^2 \SSp p(p^*-1) \int |v|^{p^*-2} |\vphi|^2   +o(\e^2).
\end{equation}

Since $v \in M$ minimizes the distance between $u$ and $\M$,  $\e\vphi = u-v $ is orthogonal (in some appropriate sense) to the tangent space of $\M$ at $v$,
which we shall see coincides with the span the first two eigenspaces of an appropriate weighted linearized $p$-Laplacian. Then, a gap in the spectrum in this operator allows us to show that 
$$ c\,\dd(u,\M)^2=c\,\e^2 \int A_v[\na \vphi, \na \vphi]  \leq\e^2 p \int A_v[\na \vphi, \na \vphi]   - \e^2\SSp p(p^*-1) \int |v|^{p^*-2} |\vphi|^2   $$
for a positive constant $c=c(n,p)$. Together with \eqref{star}, this implies
\begin{align*} \dd(u, \M)^2  + o(\e^2) \leq C\delta(u).\end{align*}
Now, if the term $o(\e^2)$ could be absorbed into $\dd(u, \M)^2$, then we could use the estimate \eqref{boundP} below to obtain
$$\int |\na u - \na v|^p \leq C\delta(u),$$
which would conclude the proof.
\subsection{Where this approach falls short} \ The problem arises exactly when trying to absorb the term $o(\e^2)$. 
Indeed, recalling that  $\e=\|\na (u-v)\|_{L^p}$, we are asking whether
$$
o(\|\na u-\na v\|_{L^p}^2) \ll \dd(u, \M)^2\approx \int |\na v|^{p-2} |\na u-\na v|^2
$$
(recall Remark \ref{rmk:dist}),
and unfortunately this is false in general.
Notice that this problem never arises in \cite{BianchiEgnell91} for the case $p=2$, as the above inequality reduces to
$$
o(\|\na u-\na v\|_{L^2}^2) \ll \|\na u-\na v\|_{L^2}^2,
$$
which is clearly true.

\subsection{The solution} A Taylor expansion of the deficit will not suffice to prove Theorem~\ref{MainThm} as we cannot hope to absorb the higher order terms. Instead, for a function $u\in \dot W^{1,p}$, we give two different expansions, each of which gives a lower bound on the deficit, by splitting the terms between the second order term and the $p^{\rm{th}}$ order term using elementary inequalities (Lemma~\ref{numbers}). Pairing this with an analysis of the second variation, we obtain the following:
\begin{proposition}\label{Bounds on the deficit} There exist constants 
 ${\bf{c}}_1, {\bf{C}}_2, $ and ${\bf{C}}_3$, depending only on $p$ and $n$, such that the following holds.
Let $u\in \dot W^{1,p}$ be a function satisfying \eqref{assumption2} and let $v\in \M$ be a function where the infimum of the distance \eqref{distance} is attained. Then
\begin{align}
\label{bound2}  {\bf{c}}_1\, \dd(u, \M)^2 -{\bf{C}}_2 \, \int |\na u - \na v  |^p &\leq \delta(u) ,\\
\label{boundP}-{\bf{C}}_3\, \dd(u, \M)^2+ \frac{1}{4} \int |\na u - \na v|^p & \leq \delta(u) .
\end{align}
\end{proposition}

Individually, both inequalities are quite weak. However, as shown in Corollary~\ref{Regime 1 and 3}, they allow us to prove Theorem~\ref{MainThm} (in fact, the stronger statement $\int|\na u - \na v|^p \leq \delta(u)$) for the set of functions $u$ such that 
\begin{equation}\label{regimes} \begin{split}
\dd(u, \M)^2&= \int A_v[\na u - \na v, \na u - \na v]
 \ll \int |\na u -\na v|^p\\
 & \text{ or } \\
\dd(u, \M)^2& = \int A_v[\na u - \na v, \na u - \na v] \gg\int |\na u - \na v|^p.
\end{split}
\end{equation}
 We are then left to consider the middle regime, where 
 \[
 \int A_v[\na u - \na v, \na u - \na v] \approx \int |\na u - \na v|^p.
 \]
We handle this case as follows.
Let  $u_t:=(1-t)u+tv$ be the linear interpolation between $u$ and $v$. Choosing $t_*$ small enough, $u_{t_*}$ falls in the second regime in \eqref{regimes}, so Theorem~\ref{MainThm} holds for $u_{t_*}$. We then must relate the deficit and distance of $u_{t_*}$ to those of $u$. 
 While relating the distances is straightforward, it is not clear for the deficits whether the estimate $\delta(u_{t_*}) \leq C \delta(u)$ holds.
Still, we can show that  
 $$
 \delta(u_{t_*}) \leq C\delta(u) +C\|v\|_{L^{p^*}}^{p-1} \| u-v\|_{L^{p^*}},
 $$
 which allows us to conclude the proof.
It is this point in the proof that introduces that term $\| u-v\|_{L^{p^*}}$ in Theorem~\ref{MainThm}, and for this reason we rely on the main theorem of \cite{ciafusmag07} to prove Corollary~\ref{MainCor}. We note that the application of \cite{ciafusmag07} is not straightforward,
since the function $v$ which attains the minimum in our setting is a priori different from the one considered there (see Section \ref{Conclude} for more details).
 
\subsection{Outline of the paper}
The paper is structured as follows. \\

\noindent In Section~\ref{Preliminaries}, we introduce the operator $\LL_v$ that will be important in our analysis of the second variation of the deficit and prove some facts about  the spectrum of this operator. We also prove some elementary but crucial inequalities in Lemma~\ref{numbers} and provide orthogonality constraints that arise from taking the infimum in \eqref{distance}.
\\

\noindent In Section~\ref{expansion}, we prove Proposition~\ref{Bounds on the deficit} by exploiting a gap in the spectrum of $\LL_v$ and using the inequalities of Lemma~\ref{numbers}.  \\

\noindent  
In Section~\ref{Conclude}, we combine Proposition~\ref{Bounds on the deficit} with an interpolation argument to obtain Theorem~\ref{MainThm}. We then apply the main result of \cite{ciafusmag07} 
in order to prove Corollary~\ref{MainCor}.  \\

\noindent In Section~\ref{Spectral}, we prove the compact embedding that shows that $\LL_v$ has a discrete spectrum and justify the use of Sturm-Liouville theory in the proof of Proposition~\ref{lem: spectral properties}. \\

\noindent Section~\ref{appendix} is an appendix in which we prove a technical claim.

\section{Preliminaries}\label{Preliminaries}
In this section, we state a few necessary facts and tools.
\subsection{The tangent space of $\M$ and the operator $\LL_v$}
The set $\M$ of extremal functions defined in \eqref{M} is an $(n+2)$-dimensional smooth manifold except at $0\in \M.$ For a nonzero $v=\veee\in \M$, the tangent space is computed to be
$$T_{v}\M  =\spann \{ v,\,  \partial_\lambda v,  \, \partial_{y^1}v,\hdots, \,\partial_{y^n}v \},$$
 where $y^i$ denotes the $i$th component of $y$ and $ \partial_\lambda v = \partial_\lambda|_{\lambda = \lambda_0} v$, $ \partial_{y^i} v = \partial_{y^i}|_{y^i= y_0^i} v$.
 
 Since the functions $v=\vone$ minimize $u \mapsto \delta(u)$ and have $\|\vone\|_{L^{p^*}}=1$, by computing the Euler-Lagrange equation one discovers that
\begin{equation}\label{pLap}- \Delta_p v  = \SSp v^{p^*-1},
\end{equation}
where the $p$-Laplacian $\Delta_p$ is defined by  $\Delta_p w: = \DIVV(|\na w|^{p-2} \na w)$. Hence, differentiating \eqref{pLap} with respect to $y^i$ or $\lambda$,
we see that
\begin{equation}\label{Linearize}
 - \DIVV (A_{v}(x) \na w)= (p^*-1) \SSp v^{p^*-2} w,\qquad w \in \spann \{ \partial_\lambda v,  \, \partial_{y^1}v,\hdots, \,\partial_{y^n}v \},
\end{equation}
where $A_v(x)$ is as defined in \eqref{Matrix A}. This motivates us to consider the weighted operator 
\begin{equation}\label{op1}\LL_{v} w := - \DIVV (A_{v}(x) \na w) v^{2-p^*}
\end{equation}
 on the space $L^2(v^{p^*-2})$, where, for a measurable weight $\om :\Rn \to \R$, we let 
 \[
 \|w\|_{L^2( \om)} = \Big( \int_{\Rn} |w|^2 \om  \Big)^{1/2},  \qquad L^2(\om) = \{ w:\R^n\to \R : \|w\|_{L^2(\om)}<\infty\}.
 \]
\begin{proposition}\label{lem: spectral properties}
The operator $\LL_v$ has a discrete spectrum $\{\alpha_i\}_{i=1}^{\infty}$, with $0<\alpha_i < \alpha_{i+1}$ for all $i$, and
\begin{align}
\alpha_1 = (p-1)\SSp, & \qquad H_1 = \spann\{ v\}, \label{lambda1}\\
\alpha_2 =(p^*-1) \SSp,& \qquad H_2  = \spann\{ \partial_{\lambda }v,\,  \partial_{y^1} v, \hdots, \, \partial_{y^n} v\},\label{lambda2}
\end{align}
where $H_i$ denotes the eigenspace corresponding to $\alpha_i$. 
\end{proposition}
In particular, Proposition~\ref{lem: spectral properties} implies that
\begin{equation}\label{eqn: tangent space}
T_v\M =   \spann \{ H_1 \cup H_2\}.
\end{equation}
The Rayleigh quotient characterization of eigenvalues implies that
\begin{equation}
\alpha_3 =\inf \bigg\{\, \frac{ \langle \LL_v w, w\rangle}{\langle w, w\rangle}\, =\, \frac{ \int A_v[\na w, \na w]  }{ \int v^{p^*-2} w^2  } :  \ \ w \perp \spann \{ H_1 \cup H_2\} \bigg\}, \label{lambda3}
\end{equation}
where orthogonality is with respect to the inner product defined by
\begin{equation}\label{ip} \langle w_1, w_2 \rangle: = \int v^{p^*-2} \,w_1\,w_2  .\end{equation}
Note that the eigenvalues of $\LL_v$ are invariant under changes in $\lambda$ and $y$.

 \begin{proof}[Proof of Proposition~\ref{lem: spectral properties}]
 The discrete spectrum of $\LL_v$ follows in the usual way after establishing the right compact embedding theorem; we show the compact embedding in Corollary~\ref{Compact embedding} and give details confirming the discrete spectrum in Corollary~\ref{discrete}. Since a scaling argument shows that the eigenvalues of $\LL_{v}$ are invariant under changes of $\lambda$ and $y$, it suffices to consider the operator $\LL = \LL_v$ for $v= v_{0,1}$, letting $A = A_{v}$. 

One easily verifies that $v$ is an eigenfunction of $\LL$ with eigenvalue $(p-1) \SSp $ and that $\partial_{\lambda} v $ and $\partial_{y^i}v$ are eigenfunctions with eigenvalue $(p^*-1)\SSp$, using \eqref{pLap} and \eqref{Linearize} repectively. 
Furthermore, since $v>0 $, it follows that $\alpha_1 = (p-1)\SSp$ is the \textit{first} eigenvalue, which is simple, so \eqref{lambda1} holds.

To prove \eqref{lambda2}, we must show that $\alpha_2 = (p^*-1)\SSp$ is the \textit{second} eigenvalue and verify that there are no other eigenfunctions in $H_2$. Both of these facts follow from separation of variables and Sturm-Liouville theory.
Indeed, an eigenfunction $\vphi$ of $\LL$ satisfies
\begin{equation}
\label{eval} \DIVV (A(x)\na \vphi) + \alpha  v^{p^*-2}\vphi=0.
\end{equation}
 Assume that $\vphi$ takes the form $\varphi (x) = Y(\theta)f(r)$, where $Y:\mathbb{S}^{n-1}\to \mathbb{R} $ and $f:\mathbb{R}\to \mathbb{R}$. 
In polar coordinates,
\begin{equation}\label{polarA}
\begin{split}
\text{div} (A(x) \na \varphi) &=(p-1) |\na v|^{p-2} \partial_{rr} \varphi 
+  \frac{(p-1)(n-1)}{r} |\na v|^{p-2}\partial_r \varphi \\ 
&+\frac{1}{r^2}|\na v|^{p-2}\sum_{j=1}^{n-1} \partial_{\theta_j \theta_j } \varphi
+
(p-1)(p-2)|\na v|^{p-4} \partial_r v\, \partial_{rr} v\, \partial_r \varphi
\end{split}
\end{equation}
(this computation is given in the appendix for the convenience of the reader). As $v$ is radially symmetric, that is, $v(x) = w(|x|)$, we introduce the slight abuse of notation by letting $v(r)$ also denote the radial component: $v(r) = w(r)$, so $v'(r) = \pa_{r} v$ and $v''(r) = \pa_{rr} v.$ From \eqref{polarA}, we see that \eqref{eval} takes form 
\begin{align*}
0&=
(p-1) |v'|^{p-2} f''(r)Y(\theta) + 
 \frac{(p-1)(n-1)}{r} |v'|^{p-2} f'(r)Y(\theta)
 \\ &
  +\frac{1}{r^2}|v'|^{p-2} f(r)\Delta_{\mathbb{S}^{n-1}}Y(\theta)
+(p-1)(p-2)|v'|^{p-4} v' v''  f'(r)Y(\theta) +\alpha v^{p^*-2}f(r)Y(\theta),
\end{align*}
which yields the system
\begin{align}\label{sphere}
&0=\Delta_{\mathbb{S}^{n-1}}Y(\theta) + \mu Y(\theta)&\qquad \text{ on } \mathbb{S}^{n-1},\\
\label{Req}
\begin{split}
0=(p-1) |v'|^{p-2} f'' + \frac{(p-1)(n-1)}{r} |v'|^{p-2} f'  -\frac{\mu}{r^2}|v'|^{p-2} f \\
+(p-1)(p-2)|v'|^{p-4} v'v''  f' +\alpha v^{p^*-2}f \end{split}
  & \qquad \text{ on } [0,\infty).
\end{align}
The eigenvalues and eigenfunctions of \eqref{sphere} are explicitly known; these are the spherical harmonics. The first two eigenvalues are $\mu_1 =0$ and $\mu_2=n-1$.\\

Taking $\mu = \mu_1= 0$ in \eqref{Req}, we claim that:\\
- $\alpha_1^1= (p-1)\SSp$ and the corresponding eigenspace is $\spann\{v\}$;\\
- $\alpha_1^2 = (p^*-1)\SSp$ with the corresponding eigenspace $\spann\{\partial_{\lambda} v\}$. 

Indeed, Sturm-Liouville theory ensures that each eigenspace is one-dimensional, and that the $i$th eigenfunction has $i-1$ interior zeros.
Hence, since $v$ (resp. $\partial_{\lambda} v$) 
solves \eqref{Req} with $\mu=0$ and $\alpha= (p-1)\SSp$ (resp. $\alpha= (p^*-1)\SSp$), having  no zeros (resp. one zero) it must be the first (resp. second) eigenfunction.\\

For $\mu_2 = n-1$, the eigenspace for \eqref{sphere} is $n$ dimensional with $n$ eigenfunctions giving the spherical components of $\partial_{y^i} v, i = 1,\hdots, n$. The corresponding equation in \eqref{Req} gives $\alpha_2^1 = (p^*-1)\SSp$. As the first eigenvalue of \eqref{Req} with $\mu = \mu_2$, $\alpha_2^1$ is simple.\\

The eigenvalues are strictly increasing, so this shows that $\alpha_1^3 >(p^*-1)\SSp$ and $\alpha_2^2 >(p^*-1)\SSp$, concluding the proof.
\end{proof}
The application of Sturm-Liouville theory in the proof above is not immediately justified because ours is a {\it singular} Sturm-Liouville problem. The proof of Sturm-Liouville theory in our setting, that is, that each eigenspace is one-dimensional and that the $i$th eigenfunction has $i-1$ interior zeros, is shown in Section~\ref{Spectral}.

\subsection{Some useful inequalities}
The following lemma contains four elementary inequalities for vectors and numbers. This lemma is a key tool for getting around the issues presented in the introduction; in lieu of a Taylor expansion, these inequalities yield bounds on the deficit by splitting the higher order terms between the second order terms and the $p^{\rm{th}}$ or $p^{*\rm{th}}$ order terms. 

\begin{lemma}\label{numbers}
Let $x,y\in \mathbb{R}^n$ and  $a,b\in \R$.  The following inequalities hold.\\

\noindent For all $\ka >0,$ there exists a constant $\C = \C(p,n,\kappa)$ such that 
\begin{equation}\label{num2} |x+y|^p \geq |x|^p + p |x|^{p-2}x \cdot y  +(1-\ka)\Bigl( \frac{p}{2}  |x|^{p-2}|y|^{2}+
\frac{p(p-2)}{2}|x|^{p-4} (x\cdot y)^2\Bigr)  - \C|y|^p.
\end{equation}

\noindent For all $\kappa>0$, there exists $\C = \C(p, \kappa)$ such that
\begin{equation}\label{num4}
|a+b|^{p^*} \leq |a|^{p^*} + p^* |a|^{p^* -2 } ab + \Big(\frac{p^* ( p^* -1)}{2} + \ka \Big) |a|^{p^*-2} |b|^2 + \C |b|^{p^*}.
\end{equation}

\noindent There exists $\C = \C(p,n)$ such that 
\begin{equation}\label{num1} |x+y|^p \geq |x|^p + p |x|^{p-2}x \cdot y - \C|x|^{p-2} |y|^2 + \frac{|y|^p}{2}.
\end{equation}

\noindent There exists $\C = \C(p)$ such that 
\begin{equation}\label{num3}
|a+b|^{p^*} \leq |a|^{p^*}+ p^*|a|^{p^* -2}ab + \C |a|^{p^*-2} |b|^2 + 2 |b|^{p^*}.
\end{equation}

\end{lemma}
 
\begin{proof}[Proof of Lemma~\ref{numbers}] We only give the proof of \eqref{num2}, as the proofs of \eqref{num4}-\eqref{num3} are analogous. 
Observe that if $p$ is an even integer or $p^*$ is an integer, these inequalities follow (with explicit constants) from a binomial expansion and splitting the intermediate terms between the second order and $p^{\rm{th}}$ or $p^{*\rm{th}}$ order terms using Young's inequality. 

Suppose \eqref{num2} fails. Then there exists $\ka>0$, $\{C_j\}\subset \R$ such that $C_j \to \infty$, and $\{x_j\}$, $\{y_j\} \subset \R^n$ such that
$$ |x_j + y_j |^p - |x_j|^p < p\, |x_j|^{p-2} x_j \cdot y_j
 +(1-\ka)\left(\frac{p }{2} |x_j|^{p-2}|y_j|^2 +
 \frac{p(p-2)}{2} |x_j|^{p-4}(x_j\cdot y_j)^2 \right)- C_j |y_j|^p.$$
If $x_j=0,$ we immediately get a contradiction. Otherwise, we divide by $|x_j|^p$ to obtain
\begin{equation}\label{number contradiction} \frac{|x_j + y_j|^p}{|x_j|^p} - 1 < p\, \frac{x_j\cdot y_j}{|x_j|^2} + 
(1-\ka)\frac{p}{2}\left( \frac{|y_j|^2}{|x_j|^2} + (p-2) \frac{(x_j\cdot y_j)^2}{|x_j|^4}\right) - C_j \frac{|y_j|^p}{|x_j|^p}.\end{equation}
The left-hand side is bounded below by $-1$, so in order for \eqref{number contradiction} to hold, $|y_j|/|x_j|$ must converge to $0$ at a sufficiently fast rate. In this case, $|y_j|$ is much smaller that $|x_j|$, so
 a Taylor expansion reveals that the left-hand side behaves like 
 \begin{equation}\label{taylor number} p\, \frac{x_j\cdot y_j}{|x_j|^2} + \frac{p}{2} \frac{|y_j|^2}{|x_j|^2 } + \frac{p(p-2)}{2}\frac{(x_j\cdot y_j)^2}{|x_j|^4}+ o\Big(\frac{|y_j|^2}{|x_j|^2 }\Big),\end{equation}
which is larger than the right-hand side, contradicting \eqref{number contradiction}.
\end{proof}
With the same proof, one can show \eqref{num4} with the opposite sign: 
For all $\kappa>0$, there exists $\C = \C(p, \kappa)$ such that
\[
|a+b|^{p^*} \geq |a|^{p^*} + p^* |a|^{p^* -2 } ab - \Big(\frac{p^* ( p^* -1)}{2} + \ka \Big) |a|^{p^*-2} |b|^2 - \C |b|^{p^*}.
\]
Therefore, applying this and  \eqref{num4} to functions $v$ and $v + \vphi $ with $\int |v|^{p^*}  = \int |v+ \vphi|^{p^*}$, one obtains
\begin{equation}\label{eqn: bound for orthog constraints}
\left| \int |v|^{p^*-2} v \vphi \right| \leq \Big(\frac{p^* ( p^* -1)}{2} + \ka \Big)\int |v|^{p^*-2} |\vphi|^2 + \C\int |\vphi|^{p^*}.
\end{equation}
\subsection{Orthogonality constraints for $u-v$}
\label{sect:orthogonal}

Given a function $u\in \dot W^{1,p}$ satisfying \eqref{assumption2},
suppose that $v= c_0v_{\lambda_0, y_0} $ is a function at which the infimum is attained in \eqref{distance}. Then
\begin{equation}\label{volume constraint} \int |u|^{p^*}   = \int |v|^{p^*}   = c_0^{p^*},\end{equation}
and the energy
\begin{equation}\label{e1}
E(v) = E(\lambda, y) = \int A_{ c_0v_{\lambda,y}}[\na u- c_0\na  v_{\lambda, y},
\na u- c_0\na  v_{\lambda, y}] ,
 \end{equation}
 arising from \eqref{distance} when $u$ is fixed,
has a critical point at $(\lambda_0, y_0)$ in the $n+1$ parameters $ \lambda$ and $y^i$, $i = 1, \hdots, n$. In other words,
\begin{equation}\label{constraints} \begin{split}
0 &=  \partial_{\lambda}|_{\lambda=\lambda_0} \int A_{ c_0v_{\lambda,y}}[\na u- c_0\na  v_{\lambda, y},
\na u- c_0\na  v_{\lambda, y}]  ,
\\
0 &=  \partial_{y^i}|_{{y^i=y_0^i}}\int A_{ c_0v_{\lambda,y}}[\na u- c_0\na  v_{\lambda, y},
\na u- c_0\na  v_{\lambda, y}] .
\end{split}
\end{equation}

We express
 $u$ as $u = v + \e \vphi$, with $\vphi$ scaled such that $\int |\na \vphi|^p = 1.$ 
Computing the derivatives in \eqref{constraints} gives
\begin{equation}\begin{split}
  \label{constraints1}\e \int A_v [\na \partial_{\lambda} v, \na \varphi ]
&= \frac{\e^2(p-2) }{2} \int |\na \varphi|^2|\na v|^{p-4}\na v \cdot \na \partial_{\lambda} v  
+ \frac{\e^2(p-2)^2}{2} \int |\na \varphi|^2|\na v|^{p-4}\pa_r v\,  \partial_{r \lambda} v   , \\
\e \int A_v[\na \partial_{y^i} v,  \na \varphi ]
&=\frac{\e^2(p-2) }{2}\int |\na \varphi|^2|\na v|^{p-4}\na v\cdot \na \partial_{y^i} v 
+\frac{\e^2(p-2)^2 }{2} \int |\na \varphi|^2|\na v|^{p-4}\pa_r v \, \partial_{r y^i} v\\&
 + \e^2 (p-2)\int |\na v|^{p-2} \pa_r  \vphi \na \vphi \cdot \pa_{y^i} \hat{r},
\end{split}
\end{equation}
where $\hat r$ is as in \eqref{Matrix A}.
Furthermore, multiplying \eqref{Linearize} by $\e \vphi$ and integrating by parts implies \begin{align*}
\SSp(p^* -1)\e  \int |v|^{p^*-2} \pl v\, \vphi & =\e \int A_v [\na \partial_{\lambda} v, \na \varphi ]
 , \\ 
 \SSp (p^* -1) \e   \int |v|^{p^*-2} \pa_{y^i}v\, \vphi & =\e \int A_v[\na \partial_{y^i} v ,\na \varphi ]
, \end{align*}
  so  \eqref{constraints1} becomes
\begin{align}
\label{lambda constraint}
\e \int |v|^{p^*-2} \pl v \,\vphi   =& \e^2 C_1\biggl[ \int |\na \varphi|^2|\na v|^{p-4}\na v \cdot \na \partial_{\lambda} v + (p-2) \int |\na \varphi|^2|\na v|^{p-4}\pa_r v \, \partial_{r\lambda} v \biggr],\\
\label{y constraint}
\e \int |v|^{p^*-2} \pa_{y^i} v\, \vphi = &\e^2  C_1\biggl[\int |\na \varphi|^2|\na v|^{p-4}\na v\cdot \na \partial_{y^i} v +
(p-2)\int |\na \varphi|^2|\na v|^{p-4}\pa_r v  \,\partial_{ry^i} v \\
&+2 \int |\na v|^{p-2} \pa_r  \vphi \na \vphi \cdot \pa_{y^i} \hat{r} \biggr] , \nonumber
\end{align}
where $C_1=\frac{(p-2)}{2(p^* -1)\SSp}$.

A Taylor expansion of the constraint \eqref{volume constraint} implies
\[
 - \e \int |v|^{p^*-2} v \vphi   = \e^2\int |v|^{p^*-2} |\vphi|^2  + o(\e^2).
 \]
However, in view of the comments in the introduction, we cannot generally absorb the term $o(\e^2)$, so this is not quite the form of the orthogonality constraint that we need. In its place, using \eqref{volume constraint} and \eqref{eqn: bound for orthog constraints}, we have
\begin{equation} \label{volume 2} 
\left|\e \int |v|^{p^*-2} v \vphi \right|  \leq  \e^2\, \frac{p^*-1 + \ka}{2} \int |v|^{p^*-2} |\vphi|^2   + \C \e^{p^*} \int |\vphi|^{p^*}
\end{equation}
for any $\ka >0$, with \C = $\C(p, n, \ka)$.

The conditions \eqref{lambda constraint}, \eqref{y constraint}, and \eqref{volume 2} show that $\vphi$ is ``almost orthogonal" to $T_{v}\M$ with respect to the inner product given in \eqref{ip}.
 Indeed, dividing through by $\e$, the inner product of $\vphi$ with each basis element of $T_{v}M$ appears on the left-hand side of \eqref{lambda constraint}, \eqref{y constraint}, and \eqref{volume 2}, while the right-hand side is $O(\e)$.  As a result of \eqref{eqn: tangent space} and $\vphi$ being ``almost orthogonal" to $T_{v}\M$,
it is shown that $\vphi$ satisfies 
a Poincar\'{e}-type inequality \eqref{Lambda3Thing}, which is an essential point in the proof of Proposition~\ref{Bounds on the deficit}.

 \begin{remark}{\rm{In \cite{BianchiEgnell91}, the analogous constraints give orthogonality rather than almost orthogonality; this is easily seen here, as taking $p=2$ makes the right-hand sides of \eqref{lambda constraint} and \eqref{y constraint} vanish.}}
\end{remark}


\section{Proof of Proposition~\ref{Bounds on the deficit} and its consequences}\label{expansion}

We prove Proposition~\ref{Bounds on the deficit} combining an analysis of the second variation and the inequalities of Lemma~\ref{numbers}. As a consequence (Corollary~\ref{Regime 1 and 3}), we show that,
up to removing the assumption \eqref{assumption2}, Theorem~\ref{MainThm} holds for the two regimes described in \eqref{regimes}.

To prove Proposition~\ref{Bounds on the deficit}, we will need two facts.
First, we want to know that the infimum in \eqref{distance} is attained, so that we
can express $u$ as $u =v + \e \vphi$ where $\int  |\na \vphi|^p   = 1$, and $\vphi$ satisfies  \eqref{lambda constraint}, \eqref{y constraint}, and \eqref{volume 2}.
Second, it will be important to know that if $\delta_0$ in \eqref{assumption2} is small enough, then $\e$ is small as well. For this reason we first prove the following:

\begin{proposition}\label{small delta} The following two claims hold.
\begin{enumerate}
\item There exists $\delta_0=\delta_0(n,p)>0$ such that if
\begin{equation}\label{assume delta small} \delta( u ) \leq \delta_0 \| \na  u\|_{L^p}^p,\end{equation}
then the infimum in \eqref{distance} is attained. In other words, there exists some $v\in \M$ with $\int |v|^{p^*} = \int |u|^{p^*}$ such that 
$$ \int A_v[\na u - \na v, \na u - \na v] = \dd(u, \M)^2.$$
\item For all $\e_0>0$, there exists $\delta_0=\delta_0(n,p,\e_0)>0$ such that if $u\in \dot W^{1,p}$ satisfies \eqref{assume delta small}, then 
 $$\e := \| \na u - \na v\|_{L^p} < \e_0$$
where 
 $v\in \M $ is a function that attains the infimum in \eqref{distance}.
\end{enumerate}
\end{proposition}
\begin{proof}
We begin by showing the following fact, which will be used in the proofs of both parts of the proposition: for all $\g>0$, there exists $\delta_0=\delta_0(n,p,\g)>0$ such that if $\delta(u) \leq \delta_0 \| \na u\|_{L^p}^p,$ then 
\begin{equation}\label{assumption3}\inf \{\|\na u - \na v\|_{L^p}: v \in \M\}\leq \g \|\na u \|_{L^p}.\end{equation}
Otherwise, for some $\g>0$, there exists a sequence $\{u_k\} \subset \dot W^{1,p}$ such that $\| \na u_k \|_{L^p} = 1$ and $\delta(u_k) \to 0$
 while $$\inf \{\| \na u_k -\na v\|_{L^p}: v \in \M\}> \g.$$
 A concentration compactness argument as in \cite{lions1985,struwe1984global} ensures that there exist sequences $\{\lambda_k\}$ and $\{y_k\}$ such that, up to a subsequence, 
 $\lambda_k^{n/p^*}u_k(\lambda_k (x-y_k)) $ converges strongly in $\dot W^{1,p}$ to some $\bar v\in \M$. 
 Since
 $$
 \gamma<\left\| \na u_k - \na\Bigl[\lambda_k^{-n/p^*} \bar v\Bigl(\frac{\cdot}{\lambda_k}+y_k\Bigr)\Bigr]\right\|_{L^p}=\left\| \na \left[\lambda_k^{n/p^*}u_k(\lambda_k (\cdot -y_k))\right] - \na \bar v\right\|_{L^p}\to 0
 $$
this gives a contradiction  for $k$ sufficiently large, hence \eqref{assumption3} holds.\\
 
\noindent{\it{Proof of (1).}} Suppose $u$ satisfies \eqref{assume delta small}, with $\delta_0$ to be determined in the proof. Up to multiplication by a constant, we may assume that $\| u\|_{L^{p^*}} =1.$ By the claim above, we may take $\delta_0$ small enough so that \eqref{assumption3} holds for $\g$ as small as needed. 

The infimum on the left-hand side of \eqref{assumption3}
is attained. Indeed,
let $\{v_k\}$ be a minimizing sequence with $v_k = c_k v_{\lambda_k, y_k}$. 
The sequences $\{c_k\}$, $\{\lambda_k\}$, $\{1/\lambda_k\}$, and $\{y_k\}$ are bounded: if $\lambda_k \to \infty$ or $\lambda_k\to 0$, then for $k$ large enough there will be little cancellation in the term $ |\na u - \na v_k|^{p}$, so that 
$$\int |\na u -  \na v_k|^{p}   \geq \frac{1}{2} \int |\na u|^p,$$
contradicting \eqref{assumption3}.
The analogous argument holds if $|y_k| \to \infty$ or $|c_k| \to \infty$.
Thus $\{c_k\}$, $\{\lambda_k\}$, $\{1/\lambda_k\}$, and $\{y_k\}$ are bounded and so, up to a subsequence,
   $(c_k, \lambda_k , y_k) \to (c_0 ,\lambda_0, y_0)$ for some $(c_0, \lambda_0, y_0) \in \R\times \R^+\times \R^n$. Since the functions $cv_{\lambda,y}$
   are smooth, decay nicely, and depend smoothly on the parameters, we deduce that $v_k \to c_0v_{\lambda_0, y_0}= \tilde{v}$ in $\dot W^{1,p}$
   (actually, they also converge in $C^k$ for any $k$), hence $\tilde{v}$ attains the infimum.

To show that the infimum is attained in \eqref{distance}, we
 obtain an upper bound on the distance by using $\bar v = \tilde{v}/\|  \tilde{v}\|_{L^{p^*}}$ as a competitor. Indeed, 
 recalling Remark \ref{rmk:dist}, it follows from H\"older's inequality that
 \begin{equation*}
\dd(u, \M)^2 \leq (p-1) \int \left| \na \bar v\right|^{p-2 } \left| \na u - \na \bar v\right|^2 \leq (p-1)S^{(p-2)/p}_{p,n} \| \na u -\na  \bar v\|_{L^p}^{2/p}.
\end{equation*}
Notice that, since $\|u\|_{L^{p^*}}=1$, it follows by \eqref{assume delta small} that $\|\na u\|_{L^p} \leq 2S_{p,n}^p$
provided $\delta_0 \leq 1/2$.
Hence, 
since
$$
\bigl| \|\bar v\|_{L^{p^*}}-1\bigl| \leq \|\bar v - u\|_{L^{p^*}} \leq S_{n,p}^{-p}\|\na \bar v -\na u\|_{L^p},
$$
it follows by  \eqref{assumption3} and the triangle inequality that 
 $\| \na u- \na \bar v\|_{L^p}  \leq C(n,p)\,\g$, therefore 
 \begin{equation}\label{distance bound}
 \dd(u, \M)^2 \leq C(n,p)\,\gamma^{2/p}.
 \end{equation}
Hence, if $\{v_k\}$ is a minimizing sequence for \eqref{distance} with $v_k = v_{\lambda_k, y_k}$ (so that $\int |v_k|^{p^*} = \int |u|^{p^*}=1$), the analogous argument
as above shows that if either of the sequences $\{\lambda_k\}$,  $\{1/\lambda_k\}$, or $\{y_k\}$ are unbounded, then 
$$\dd(u, \M)^2 \geq \frac{1}{2},$$
 contradicting \eqref{distance bound} for $\g$ sufficiently small. This implies that
$v_k \to v_{\lambda_0, y_0}$ in $\dot W^{1,p}$, and by continuity $v_{\lambda_0, y_0}$ attains the infimum in \eqref{distance}.\\

\noindent{\it{Proof of (2).}} We have shown that \eqref{assumption3} holds for $\delta_0$ sufficiently small. Therefore, we need only to show that, up to further decreasing $\delta_0$, there exists $C=C(p,n)$ such that 
$$ \|\na u - \na v_0 \|_{L^p} \leq C \inf \{ \| \na u - \na v\|_{L^p}: v\in \M\},$$
where $v_0\in \M$ is the function where the infimum is attained in \eqref{distance}.

Suppose for the sake of contradiction that there exists a sequence $\{u_j\}$ such that $\delta(u_j) \to 0$ and $\|\na u_j\|_{L^p}=1$ but\begin{equation}\label{contradiction} \int|\na u_j - \na v_j|^{p} \geq j \int |\na u_j -   \na \bar v_j|^{p}, \end{equation}
where $v_j, \bar v_j\in \M$ are such that
$$
 \int A_{v_j}[  \na u_j - \na v_j ,\na u_j - \na v_j]
   = \dd(u_j, \M)^2
   $$
   and 
$$ \int|\na u_j-  \na \bar v_j |^p = \inf\Big\{  \int|\na u_j - \na v_j|^p \, :\,  v \in \M \Big\}.
$$
Since $\delta(u_j) \to 0$, the same concentration compactness argument as above implies that there exist sequences $\{ \lambda_j\} $ and $\{y_j\}$ such that, up to a subsequence, 
$\lambda_j^{n/p^*}u_j(\lambda_j (x-y_j)) $ converges in $\dot W^{1,p}$ to some  $v \in \M$ with $\|\na v\|_{L^p} = 1$. By an argument analogous to that in part $(1)$, we determine that  $v_j \to v$ in $C^k$ and $\bar v_j \to v$ in $C^k$ for any $k$.
Let
$$ \phi_j = \frac{ u _j - v_j}{\|\na u_j -  \na v_j\|_{L^{p}}} \quad \ \text{ and }  \quad \ \bar \phi_j =  \frac{ u _j - \bar v_j}{\|\na u_j - \na v_j\|_{L^{p}}}.$$
Then \eqref{contradiction} implies that
\begin{equation} \label{goes to 0 a} 1 = \int |\na \phi_j |^{p}   \geq j \int | \na\bar \phi_j|^{p}  . \end{equation}
In particular, $\na \bar \phi_j \to 0 $ in $L^{p}.$
Now define $$\psi_j  =\phi_j - \bar \phi_j = \frac{ \bar v_j - v_j}{\|\na u_j - \na v_j\|_{L^{p}}}.$$
For any $\eta>0$, \eqref{goes to 0 a} implies that $1-\eta \leq \|\na\psi_j\|_{L^p} \leq 1+ \eta$ for $j$ large enough. In particular, $\{\na \psi_j\}$ is bounded in $L^{p}$ and so $\na \psi_j \rightharpoonup \na \psi $ in $L^{p}$ for some $\psi\in \dot W^{1,p}$.

We now consider the finite dimensional manifold $\bar \M := \{ v - \bar v : v, \bar v \in \M\}$.
Since $v_j, \bar v_j \to v,$ the sequences $\{\lambda_j\},\{1/\lambda_j\}, \{ y_j\}, \{\bar \lambda_j\},\{1/\bar \lambda_j\}$ and $\{\bar y_j\}$ are contained in some compact set, and thus all norms of $\bar v_j - v_j$ are equivalent: for any norm $ \vertiii{\cdot}$ on $\bar \M$ there exists $\mu>0$ such that 
\begin{equation} \label{all norms equivalent} \mu \| \na\bar v_j - \na v_j \|_{L^{p}} \leq  \vertiii{\na \bar v_j -\na v_j} \leq \frac{1}{\mu}  \| \na\bar v_j -\na v_j \|_{L^{p}}.\end{equation}
Dividing \eqref{all norms equivalent} by $\|\na u_j - \na v_j\|_{L^{p}}$ gives
 \begin{equation} \label{all norms equivalent 2}
  \mu(1 - \eta )\leq \mu \|\na \psi_j \|_{L^{p}}  \leq \vertiii{\na\psi_j}\leq \frac{1}{\mu}\| \na \psi_i\|_{L^p}  \leq \frac{1 + \eta}{\mu}.\end{equation}
 Taking the norm $\vertiii{\cdot} = \| \cdot \|_{C^k}$, the upper bound in \eqref{all norms equivalent 2} and the Arzel\`{a}-Ascoli theorem imply that $\psi_j$ converges, up to a subsequence, to $\psi$ in $C^k$. The lower bound in \eqref{all norms equivalent 2} implies that $\| \psi\|_{C^k} \neq 0$.

To get a contradiction, we use the minimality of $v_j$ for $\dd(u_j, \M)$ to obtain
\begin{equation}\label{square}\begin{split}
\int |\na \bar v_j |^{p-2} &|\na \bar \phi_j|^2+(p-2) \int |\na \bar v_j |^{p-2} |\pa_r \bar \phi_j|^2  \geq \int |\na v_j|^{p-2} |\na \phi_j|^2+ (p-2) \int |\na v_j|^{p-2} |\pa_r \phi_j|^2 \\
& = \int |\na v_j|^{p-2} |\na \bar \phi_j|^2 +2 \int |\na v_j|^{p-2} \na \bar\phi_j \cdot \na \psi_j + \int |\na v_j|^{p-2} |\na \psi_j|^2\\
&+ (p-2)\left( \int |\na v_j|^{p-2} |\pa_r \bar \phi_j|^2 +2 \int |\na v_j|^{p-2} \pa_r \bar\phi_j  \pa_r \psi_j + \int |\na v_j|^{p-2} |\pa \psi_j|^2\right).
\end{split}\end{equation}
Since 
$$ \int |\na \bar v_j |^{p-2} |\na \bar \phi_j|^2  -  \int |\na v_j|^{p-2} |\na \bar \phi_j|^2 \to 0$$
and 
$$ \int |\na \bar v_j |^{p-2} |\pa_r \bar \phi_j|^2  -  \int |\na v_j|^{p-2} |\pa_r\bar \phi_j|^2 \to 0,$$
\eqref{square} implies that
\begin{equation} \label{contra1}
\begin{split} 
0 &\geq 2\underset{j\to \infty}\lim\int |\na v_j|^{p-2} \na \bar\phi_j \cdot \na \psi_j +\underset{j \to \infty}{\lim} \int |\na v_j|^{p-2} |\na \psi_j|^2\\
&+(p-2)\left(2\underset{j\to \infty}\lim \int |\na v_j|^{p-2} \pa_r \bar\phi_j  \pa_r \psi_j +
\underset{j\to \infty}\lim \int |\na v_j|^{p-2} |\pa \psi_j|^2\right)
.
\end{split}
\end{equation}
However, since $\na \bar \phi_j \to 0 $ in $L^{p}$,
$$
\underset{j\to \infty}\lim\int |\na v_j|^{p-2} \na \bar\phi_j \cdot \na \psi_j  =0 \quad \text{ and } \quad 
\underset{j\to \infty}\lim \int |\na v_j|^{p-2} \pa_r \bar\phi_j  \pa_r \psi_j  =0.
$$
In addition, the terms 
$$
  \int |\na v_j|^{p-2} |\na \psi_j|^2 \quad \text{ and } \int |\na v_j|^{p-2} |\pa_r \psi_j|^2
$$
 converge to something strictly positive, as $\psi_j \to \psi \not\equiv 0$ and $v_j \to v$ with $\nabla v(x) \neq 0$ for all $x\neq 0$.
 This contradicts \eqref{contra1} and concludes the proof. 
\end{proof}
The following Poincar\'{e} inequality will be used in the proof of Proposition~\ref{Bounds on the deficit}:
\begin{lemma}
There exists a constant $C>0$ such that 
 \begin{equation}\label{Poincare}
 \int |v|^{p^*-2}|\vphi|^2\leq C \int |\na v|^{p-2}|\na \vphi|^2
 \end{equation}
for all $\vphi \in \dot W^{1,p}$.
\end{lemma}
\begin{proof}
Let $v\in \M$ and $\vphi \in C_0^{\infty}.$ As $v$ is a local minimum of the functional $\delta$,
 \begin{align*} 0& \leq   \frac{d^2}{d\e^2 }\bigg|_{\e=0}\, \delta(v+ \e \vphi) = 
 p \int |\na v|^{p-2}|\na  \varphi|^2 + p(p-2)\int |\na v|^{p-2}|\pa_r  \varphi|^2\\
&- \SSp \left( p\Big(\frac{p}{p^*} - 1 \Big) \Big(\int|v|^{p^*}\Big)^{p/p^* -2} \Big(\int v^{p^*-2} v\,\vphi\Big)^2
+ p(p^*-1) \big(\int |v|^{p^*}  \big)^{p^*/p -1} \int |v|^{p^*-2} \vphi^2\right).
\end{align*}
Noting that 
\[
 \int |\na v|^{p-2}|\pa_r  \varphi|^2 \leq \int |\na v|^{p-2}|\na  \varphi|^2 
 \qquad \text{ and } \qquad
\Big(\int|v|^{p^*}\Big)^{p/p^* -2} \Big(\int v^{p^*-2} v\,\vphi\Big)^2 \geq 0,
\]
 this implies that
 $$0 \leq  p(p-1) \int |\na v|^{p-2}|\na  \varphi|^2 -\SSp p(p^*-1) \Big(\int |v|^{p^*}  \Big)^{p^*/p -1} \int |v|^{p^*-2} \vphi^2.$$
Thus \eqref{Poincare} holds for $\vphi \in C_0^{\infty}$, and for $\vphi\in \dot W^{1,p}$ by approximation.
 \end{proof}

We now prove Proposition~\ref{Bounds on the deficit}.
\begin{proof}[Proof of Proposition~\ref{Bounds on the deficit}]
First of all, thanks to \eqref{assumption2}, we can apply Proposition~\ref{small delta}(1) to ensure that some $v= c_0v_{\lambda_0, y_0} \in \M$ attains the infimum in \eqref{distance}. 
Also, expressing $u$ as $u =v + \e \vphi$ where $\int  |\na \vphi|^p   = 1$, it follows from Proposition~\ref{small delta}(2) and the discussion in Section \ref{sect:orthogonal}
that $\e$ can be assumed to be as small as desired (provided $\delta_0$ is chosen small enough) and that $\vphi$ satisfies \eqref{lambda constraint}, \eqref{y constraint}, and \eqref{volume 2}. 
Note that, since all terms in \eqref{bound2} and \eqref{boundP} are $p$-homogeneous, without loss of generality we may take $c_0 = 1.$\\

\noindent{\it{Proof of \eqref{bound2}.}}
The inequalities \eqref{num2} and \eqref{num4} are used to expand the gradient term and the function term in $\delta(u)$ respectively, splitting higher order terms between the second order and the $p^{\rm{th}}$ or $p^{*\rm{th}}$ order terms.

From \eqref{num2} and for $\ka=\ka(p,n)>0$ to be chosen at the end of the proof, we have
\begin{equation}\label{na u exp}\begin{split} \int |\na u |^p & \geq
  \int |\na v|^p + \e p \int |\na v|^{p-2} \na v \cdot \na \varphi   \\
  & +\frac{\e^2 p (1-\ka)}{2}\Bigl(\int |\na v|^{p-2}|\na  \varphi|^2 + (p-2)\int |\na v|^{p-2}|\pa_r  \varphi|^2\Big) -\e^p\,  \C \int |\na \vphi|^p.
  \end{split}\end{equation}
Note that the second order term is precisely $\frac{1}{2}\e^2 p (1-\ka) \int A_v[\na \vphi, \na \vphi].$
Similarly, \eqref{num4} gives
\begin{equation}\label{exp12}\begin{split}
\int |u|^{p^*} & \leq 1 + \e p^* \int v^{p^*-1 } \varphi 
+   \e^2\Big(\frac{p^*(p^*-1)}{2}+\frac{p^*\ka}{2\SSp}\Big) \int v^{p^*-2}\varphi^2 +\C\e^{p^*}  \int |\vphi|^{p^*}.
\end{split}
\end{equation}
 From the identity \eqref{pLap}, the first order term in  \eqref{exp12} is equal to
 \begin{align}\label{identity} 
\e p^* \int v^{p^*-1 } \varphi  & 
 = \e p^* \SSpm \int |\na v|^{p-2} \na v\cdot \na \varphi  .
\end{align}
Using \eqref{identity} and recalling that $(p^*-1) \SSp = \alpha_2$ (see \eqref{lambda2}),  \eqref{exp12}  becomes
 \begin{align*}
 \int |u|^{p^*}  
 \leq  1 & +  \frac{\e p^*}{ \SSp} \int |\na v|^{p-2} \na v\cdot \na \varphi   +\frac{ \e^2 p^*(\alpha_2+\ka)}{2\SSp}  \int v^{p^*-2} \varphi^2  +\C\e^{p^*},
 \end{align*}
The following estimate holds, and is shown below:
\begin{align}\label{Lambda3Thing}
\e^2 \int v^{p^*-2}\vphi^2   &\leq (1+2\ka) \frac{\e^2}{\alpha_3}\int A_v[\na \vphi, \na \vphi ]
+\C\e^p,
\end{align}
Philosophically, \eqref{Lambda3Thing} follows from a spectral gap analysis, using \eqref{lambda3} and the fact that \eqref{lambda constraint},  \eqref{y constraint}, and \eqref{volume 2} imply that $\vphi$ is ``almost orthogonal" to $H_1$ and $H_2$.

As $\e$ may be taken as small as needed, using \eqref{Lambda3Thing}
 we have
\begin{equation*}
 \int |u|^{p^*}  \leq  1 +  \frac{p^* }{\SSp} \bigg(\e\int |\na v|^{p-2} \na v\cdot \na \varphi  
 +\frac{\e^2(\alpha_2 +\ka)(1+2 \ka)}{2\alpha_3}\int A_v[\na \vphi, \na \vphi] + \C\e^{p} \bigg).
\end{equation*}
The function $z \mapsto |z|^{p/p^*}$ is concave, so $\|u\|_{L^{p^*}}^p \leq1 +\frac{p}{p^* }(\int |u|^{p^*} - 1)$: 
\begin{multline}\label{u exp}
\SSp\|u\|_{L^{p^*}}^p  
 \leq \SSp +  p  \bigg(\e\int |\na v|^{p-2} \na v\cdot \na \varphi 
 +\frac{\e^2(\alpha_2 +\ka)(1+2 \ka)}{2\alpha_3}\int A_v[\na \vphi, \na \vphi]  + \C\e^{p} \bigg).
\end{multline}
Subtracting \eqref{u exp} from \eqref{na u exp} gives  
\begin{align*} 
\delta(u) \geq 
\frac{\e^2p}{2} \left(1 - \ka - \frac{(\alpha_2 +\ka)(1+2 \ka)}{\alpha_3}\right)\int A[\na \vphi, \na \vphi]-\C \e^p.
\end{align*}
Since $1-\frac{\alpha_2}{\alpha_3}>0$, we may choose  $\ka$ sufficiently small so that  $1 -\ka-\frac{(\alpha_2 +\ka)(1+2\ka)}{\alpha_3}>0$. To conclude the proof of \eqref{bound2}, we need only to prove \eqref{Lambda3Thing}.
\\

\noindent{\it Proof of \eqref{Lambda3Thing}.}
If $\vphi$ were orthogonal to $T_v\M$ instead of almost orthogonal, that is, if the right-hand sides of  \eqref{lambda constraint},  \eqref{y constraint}, and \eqref{volume 2} were equal to zero, then \eqref{Lambda3Thing} would be an immediate consequence of \eqref{lambda3}. Therefore, the proof involves showing that the error in the orthogonality relations is truly higher order, in the sense that it can be absorbed in the other terms.
 
Up to rescaling $u$ and $v$, we may assume that $\lambda_0 =1$ and $y_0 = 0$. We recall the inner product  $\langle w, y \rangle$ defined in \eqref{ip} which gives rise to the norm 
$$
\|w\| =\Big( \int |v|^{p^*-2} w^2  \Big)^{1/2}.
$$
As in Section~\ref{Preliminaries}, 
we let $H_i$ denote the eigenspace of $\LL_v$ in $L^2( v^{p^*-2})$ corresponding to eigenvalue $\alpha_i$, so $H_i = \spann\{ Y_{i,j}\}_{j=1}^{N(i)},$ where $Y_{i,j}$ is an eigenfunction with eigenvalue $\alpha_i$ with $\| Y_{i,j}\| = 1.$
We express $\e\vphi$ in the basis of eigenfunctions:
$$
\e\vphi = \sum_{i=1}^{\infty} \sum_{j=1}^{N(i)} \beta_{i,j} Y_{i,j} \qquad {\rm{where}} \qquad\beta_{i,j}:=\e \int |v|^{p^*-2}  \vphi Y_{i,j}. 
$$
We let $\e\tilde{\vphi}$ be the truncation of $\e\vphi$:
$$
\e\tilde{\vphi}  = \e\vphi  - \sum_{i=1}^{2}\sum_{j=1}^{N(i)} \beta_{i,j} Y_{i,j},
$$
so that $\tilde{\vphi}$ is orthogonal to $\spann \{ H_1 \cup H_2\}$ and, introducing the shorthand
$\beta_i^2 : = \sum_{j=1}^{N(i)} \beta_{i,j}^2$,
\begin{equation}\label{norm relation} \int |v|^{p^*-2} (\e \vphi)^2    =   \int |v|^{p^*-2}(\e \tilde{\vphi})^2   +\beta_1^2 + \beta_2^2 .\end{equation}
Applying \eqref{lambda3} to $\tilde{\vphi}$ implies that 
$$ \int |v|^{p^*-2}(\e \tilde{\vphi})^2 \leq \frac{\e^2}{\alpha_3} \langle \LL_v\tilde{\vphi}, \tilde{\vphi}\rangle,$$
which combined with \eqref{norm relation} gives
\begin{equation}\begin{split}\label{L3fixA}
 \int |v|^{p^*-2} (\e \vphi)^2     
 \leq \frac{\e^2}{\alpha_3} \langle \LL_v\tilde{\vphi}, \tilde{\vphi}\rangle + \beta_1^2  +\beta_2^2
&=\frac{1}{\alpha_3} \sum_{i=3}^{\infty}\alpha_i \beta_i^2  + \beta_1^2 + \beta_2^2\\
& \leq \frac{\e^2}{\alpha_3}\langle \LL_v \vphi, \vphi\rangle +    \Bigl(1- \frac{\alpha_1}{\alpha_3}\Bigr)(\beta_1^2 + \beta_2^2).
  \end{split}
\end{equation}
We thus need to estimate $\beta_1^2 + \beta_2^2.$ 
The constraint \eqref{volume 2} implies
\begin{align*} \beta_1^2& \leq \Big(\e^2\,\frac{p^*-1 + \ka}{2}  \int |v|^{p^*-2} |\vphi|^2   + \C \e^{p^*} \int |\vphi|^{p^*} \Big)^2\\
&\leq \C \e^4 \Big(\int |v|^{p^*-2} |\vphi|^2\Big)^2 + \C \e^{2p^*} \Big(\int |\vphi|^{p^*} \Big)^2.
\end{align*}
By \eqref{Poincare}, $\int |v|^{p^*-2} |\vphi|^2 \leq \int \na v|^{p-2} |\na \vphi|^2.$
Furthermore, both 
$\int |\na v|^{p-2} |\na \vphi|^2$ and $\int |\vphi|^{p^*} $ are universally bounded, so for $\e$ sufficiently small depending only on $p$ and $n$ and $\ka$,
\begin{equation}\label{beta 1 bound} 
\beta_1^2 \leq \frac{\ka \e^2 }{\alpha_3}\Big( \int |\na v|^{p-2} |\na \vphi|^2 +(p-2) \int |\na v|^{p-2} |\pa_r \vphi|^2\Big)   + \C \e^{p} .
\end{equation}
For $\beta_{2,1}^2$, we notice that H\"{o}lder's inequality and \eqref{lambda constraint} imply
\begin{equation}\label{beta bound} \begin{split}
\beta_{2,1}^2  &
\leq \Big(C_{p,n}\e^2 \int |\na v|^{p-3} | \na \vphi |^2 \frac{|\na \pl v| }{\| \pl v\|} \Big)^2 \\
&  \leq C_{p,n} \frac{\int |\na v|^{p-2} |\na \pl v|^2 }{\|\pl v\|^2 }  \int |\na v|^{p-4} |\e\na  \vphi|^4
=C_{p,n} \e^4 \int |\na v|^{p-4} | \na \vphi|^4,
 \end{split}\end{equation}
 where the final equality follows because the term $\int |\na v|^{p-2} |\na \pl v|^2/\|\pl v\|^2$ is bounded (in fact, it is bounded by $\alpha_2$). Then, using Young's inequality, we get
 $$\beta_{2,1}^2 
 \leq \frac{\e^2 \ka}{(n+1)\alpha_3}\Big(\int |\na v|^{p-2 } | \na \vphi|^2+(p-2)\int |\na v|^{p-2} |\pa_r \vphi|^2\Big)
    + C_{\ka, p}\e^p \int |\na \vphi |^p.$$

The analogous argument using \eqref{y constraint} implies that 
\begin{equation}\label{ok}
\beta_{2,j}^2  \leq C_{p,n} \e^4 \int |\na v|^{p-4} | \na \vphi|^4
+ C_{p,n}\e^4\left(\int |\na v|^{p-2} \pa_r  \vphi \na \vphi \cdot \frac{\pa_{y^i} \hat{r}}{\| \pa_y v\|} \right)^2.
\end{equation}
 for $j=2,\hdots, n+1$. 
 For the second term in \eqref{ok},
  H\"{o}lder's inequality implies that 
$$\left(\int |\na v|^{p-2} \pa_r  \vphi \na \vphi \cdot \frac{\pa_{y^i} \hat{r}}{\| \pa_y v\|} \right)^2
\leq 
\int |\na v|^{p-4} |\na \vphi |^4 \int |\na v|^{p} \frac{|\pa_{y^i} \hat{r}|^2}{\| \pa_{y^i} v\|^2}.$$
Since 
\begin{align*} \pa_{y^i} \hat{r} = \frac{x^ix}{|x|^3}, \qquad |\pa_{y^i} \hat{r}| \leq \frac{1}{|x|},
\end{align*}
we find that $\int |\na v|^{p} \frac{|\pa_{y^i} \hat{r}|^2}{\| \pa_{y^i} v\|^2}$ converges, so \eqref{ok} implies that 
\begin{equation*}
\beta_{2,j}^2  \leq C_{p,n} \e^4 \int |\na v|^{p-4} | \na \vphi|^4.
\end{equation*}
Then using Young's inequality just as in \eqref{beta bound}, we find that 
 $$\beta_{2,j}^2 
 \leq \frac{\e^2 \ka}{(n+1)\alpha_3}\Big(\int |\na v|^{p-2 } | \na \vphi|^2 +(p-2)\int |\na v|^{p-2} |\pa_r \vphi|^2\Big)
    + C_{\ka, p}\e^p \int |\na \vphi |^p,$$
and thus
\begin{equation}\label{beta 2 bound}
\beta_{2}^2\leq \frac{\e^2 \ka}{\alpha_3}\Big(\int |\na v|^{p-2 } | \na \vphi|^2 +(p-2)\int |\na v|^{p-2} |\pa_r \vphi|^2\Big)
  + C_{\ka, p}\e^p .
  \end{equation}
Together \eqref{L3fixA}, \eqref{beta 1 bound}, and \eqref{beta 2 bound} imply \eqref{Lambda3Thing}, as desired.
\\

\noindent \textit{Proof of \eqref{boundP}.}
The proof of \eqref{boundP} is similar to, but simpler than, the proof of \eqref{bound2}, as no spectral gap or analysis of the second variation is needed. The principle of the expansion is the same, but now we use \eqref{num1} and \eqref{num3} for the expansion, putting most of the weight of the higher order terms on the second order term and preserving the positivity of the $p^{\rm{th}}$ order term.

From \eqref{num1}, we have 
\begin{equation}\label{nauP}
 \int |\na u |^p   \geq \int |\na v|^p  + p\e \int |\na v|^{p-2} \na v \cdot \na \vphi   -\C\, \e^2 \int |\na v|^{p-2}|\na  \varphi|^2  + \frac{\e^p}{2} \int |\na \vphi|^p .
 \end{equation}
 Similarly, \eqref{num3} implies
 \begin{equation}\label{exp1a}\begin{split}
\int |u|^{p^*}   \leq 1
+ \e p^* &\int v^{p^*-1 } \varphi  
+ \C \, \e^2 \int v^{p^*-2}\varphi^2  + 2\e^{p^*} \int |\vphi|^{p^*} .
\end{split}
\end{equation}
As before,
 the identity \eqref{pLap} implies \eqref{identity}, so \eqref{exp1a} becomes 
 \begin{align*}
 \int |u|^{p^*}  
 \leq 1 & +  \e p^* \SSpm \int |\na v|^{p-2} \na v\cdot \na \varphi  + \C \, \e^2 \int v^{p^*-2}\varphi^2  + 2 \e^{p^*}\int |\vphi|^{p^*} .
 \end{align*}
By the Poincar\'{e} inequality \eqref{Poincare},
\begin{align*}\int |u|^{p^*}  & \leq  1 + \e p^* \SSpm \int |\na v|^{p-2} \na v\cdot \na \varphi  + \C \e^2 \int|\na v|^{p-2}  |\na \varphi|^2 + 2\e^{p^*} .
\end{align*} 
As in \eqref{u exp}, the concavity of $z \mapsto |z|^{p/p^*}$ yields 
\begin{equation}\label{uexp}\begin{split}
\SSp\|u\|_{L^{p^*}}^p & \leq \SSp +\e p \int |\na v|^{p-2} \na v\cdot \na \varphi  +\C \e^2 \int|\na v|^{p-2}  |\na \varphi|^2 + \C \e^{p^*}.
\end{split}
\end{equation}
Subtracting \eqref{uexp} from  \eqref{nauP} gives
\begin{align*} \delta(u) &\geq 
 -\C \, \e^2 \int |\na v|^{p-2}|\na  \varphi|^2  + \frac{\e^p}{2} -\C \e^{p^*}
 \\
 & \geq -\C\dd(u,M)^2  + \frac{\e^p}{4}.
 \end{align*}
The final inequality follows from Remark~\ref{rmk:dist} and once more taking $\e$ is as small as needed. This concludes the proof of \eqref{boundP}.
\end{proof}

\begin{corollary} \label{Regime 1 and 3} Suppose $u\in \dot W^{1,p}$ is a function satisfying \eqref{assumption2} and $v\in \M$ is a function where the infimum in \eqref{distance} is attained.  There exist constants $\C_*, \cc_*$ and $c$, depending on $n$ and $p$ only, such that if 
\begin{equation}\label{extreme} 
\C_* \leq \frac{\int A_v [\na u  - \na v, \na u - \na v]  }{\int |\na u - \na v|^p  } \ \ \ \text{ or } \ \ \ \cc_* \geq \frac{\int A_v [\na u  - \na v, \na u - \na v]  }{\int |\na u - \na v|^p  },
\end{equation}
then
$$c \int |\na u - \na v|^p \leq  \delta(u).$$
\end{corollary}

\begin{proof} 
Let $\C_* = \frac{2\C_2}{\cc_1}$ and let $\cc_* = \frac1{8\C_3}$ where $\cc_1, \C_2$ and $\C_3$ are as defined in Proposition~\ref{Bounds on the deficit}. 
First suppose that $u$ satisfies the first condition in \eqref{extreme}. Then in \eqref{bound2}, we may absorb the term $\C_2 \int |\na u - \na v|^p $ into the term $\cc_1 \dd(u, \M)^2,$ giving us 
$$ \frac{\cc_1}{2} \, \dd(u, \M)^2 \leq \delta(u).$$
Given this control, we may bootstrap using \eqref{boundP} to gain control of the stronger distance:
$$\frac{1}{4} \int |\na u - \na v|^p \leq \delta(u) + \C_3\,  \dd(u, \M)^2 \leq \C \delta(u).$$

Similarly, if $u$ satisfies the second condition in \eqref{extreme}, then we may absorb the term
$\C_3 \, \dd(u, \M)^2$ into the term $\frac{1}{4} \int |\na u - \na v|^p$ in \eqref{boundP}, giving us
 $$\frac{1}{8} \int |\na u - \na v|^p \leq  \delta(u) .$$
 \end{proof}

\section{Proofs of Theorem~\ref{MainThm} and Corollary~\ref{MainCor}}\label{Conclude}
Corollary~\ref{Regime 1 and 3} implies Theorem~\ref{MainThm} for the functions $u\in \dot W^{1,p}$ that satisfy \eqref{assumption2} and that lie in one of the two regimes described in \eqref{regimes}. Therefore, to prove Theorem~\ref{MainThm}, it remains to understand the case when the terms $\int A_v [\na u  - \na v, \na u - \na v]$ and $\int |\na u - \na v|^p$ are comparable and to remove the assumption \eqref{assumption2}. The following proposition accomplishes the first.
 
\begin{proposition}\label{Regime 2}
Let $u\in \dot W^{1,p}$ be a function satisfying \eqref{assumption2}, and let $v \in \M$ be a function where the infimum in \eqref{distance} is attained. If
\begin{equation}\label{bad regime}\cc_* \leq \frac{\int A_v [\na u  - \na v, \na u - \na v]  }{\int |\na u - \na v|^p  }\leq \C_*,\end{equation}
where $\cc_*$ and $\C_*$ are the constants from the Corollary~\ref{Regime 1 and 3}, then
\begin{equation}\label{bad regime ok}\int |\na u -\na v|^p \leq  C\delta(u) + C\|v\|_{L^{p^*}}^{p-1}\| u -v\|_{L^{p^*}}
\end{equation}
for a constant $C$ depending only on $p$ and $n$.
\end{proposition}

 \begin{proof}
 Suppose $u$ lies in the regime \eqref{bad regime}. Then we consider the linear interpolation $u_t := tu + (1-t) v$
 and notice that 
$$\frac{\int A_v [\na u_t  - \na v, \na u_t - \na v]}{\int | \na u_{t}-\na v|^p } 
= \frac{ t^{2}\int A_v [\na u  - \na v, \na u - \na v]}{t^p \int | \na u-\na v |^p }\geq t^{2-p}\cc_*.$$
Hence, there exists  $t_*$ sufficiently small, depending only on $p$ and $n$, such that $t_*^{2-p}\cc_* > \C_*$. 

We claim that we may apply Corollary~\ref{Regime 1 and 3} to $u_{t_*}$. 
This is not immediate because $v$ may not attain the infimum in \eqref{distance} for $u_{t_*}$. However, each step of the proof holds if we expand $u_{t_*}$ around $v$. Indeed, keeping the previous notation of $ u - v=\e \vphi $ with $\int |\na \vphi|^p=1$, we have $ u_{t_*} - v=t_*\e \vphi $. so the orthogonality constraints in \eqref{lambda constraint}, \eqref{y constraint}, and \eqref{volume 2} still hold for $u_{t_*}$ and $v$ by simply multiplying through by $t_*$ (this changes the constants by a factor of $t_*$ but this does not affect the proof). Furthermore, \eqref{assumption2} is used in the proofs of Proposition~\ref{Bounds on the deficit} and \eqref{Lambda3Thing} to ensure that $\e$ is a small as needed to absorb terms.  Since $t_*<1$, if $\e$ is sufficiently small then so is $t_* \e$. With these two things in mind, every step in the proof of Proposition~\ref{Bounds on the deficit}, and therefore Corollary~\ref{Regime 1 and 3} goes through for $u_{t_*}$.

Corollary~\ref{Regime 1 and 3} then implies that
$$   t_*^p  \int  |\na u - \na v|^p=\int |\na u_{t_*} - \na v|^p  \leq  C\delta(u_{t_*}).$$
Therefore, \eqref{bad regime ok} follows if we can show
\begin{equation}\label{bound def t}
\delta(u_{t_*}) \leq C\delta(u) +C\|v\|_{L^{p^*}}^{p-1} \|u - v\|_{L^{p^*}}.
\end{equation}
In the direction of \eqref{bound def t}, by convexity and recalling that $\|\na v\|_{L^p} = \SSS \|v\|_{L^{p^*}} = \SSS \| u \|_{L^{p^*}},$ we have
\begin{equation}\label{convexity 1} \begin{split}
\delta(u_{t_*}) &= \int |t_* \na u + (1-t_* ) \na v|^p   - \SSp\, \|t_* u + (1-t_* ) v\|_{L^{p^*}}^{p}\\
& \leq  t_* \int | \na u |^p+ (1-t_* )\int | \na v|^p   - \SSp\|t_* u + (1-t_* ) v\|_{L^{p^*}}^{p} \\
&= t_* \,\delta(u)+  \SSp\,\left(\| v\|_{L^{p^*}}^p - \|t_* u + (1-t_*) v\|_{L^{p^*}}^{p}\right).
\end{split}
\end{equation}
Also, by the triangle inequality,
$$  \| t_*(u-v) + v\|_{L^{p^*}}^p 
 \geq (\| v\|_{L^{p^*}} -\| t_*(u-v)\|_{L^{p^*}})^p,$$
and by the convexity of the function $f(z) = |z|^p$, 
$f(z+y) \geq f(z) + f'(z) y$,  and so \begin{align*} 
 (\| v\|_{L^{p^*}} -\| t_*(u-v)\|_{L^{p^*}})^p \geq  \| v\|_{L^{p^*}} -p \| v\|_{L^{p^*}}^{p-1}\| u-v\|_{L^{p^*}} .
 \end{align*}
 These two inequalities imply that
 $$\| v\|_{L^{p^*}}^p - \|t_* u + (1-t_*) v\|_{L^{p^*}}^{p} \leq p \| v\|_{L^{p^*}}^{p-1}\| u-v\|_{L^{p^*}} .
$$
Combining this with \eqref{convexity 1} yields \eqref{bound def t}, concluding the proof.
\end{proof}
From here, the proof of Theorem~\ref{MainThm} follows easily:
\begin{proof}[Proof of Theorem~\ref{MainThm}]

Together, Corollary~\ref{Regime 1 and 3} and Proposition~\ref{Regime 2} imply the following: there exists some constant $C$ such that if $u\in \dot W^{1,p}$ satisfies \eqref{assumption2}, then there is some $v\in \M$ such that
$$\int |\na u - \na v|^p \leq C \delta(u) + C\|v\|_{L^{p^*}}^{p^*-1}\|u - v\|_{L^{p^*}}.$$
Therefore, we need only to remove the assumption \eqref{assumption2} in order to complete the proof of Theorem~\ref{MainThm}. However, in the case where \eqref{assumption2} fails, then trivially,
$$\inf \{ \|\na u - \na v\|_{L^p}^p : v\in \M \} \leq \| \na u \|_{L^p}^p \leq \frac{1}{\delta_0}\delta(u) .$$
Therefore, by choosing the constant to be sufficiently large, Theorem~\ref{MainThm} is proven.
\end{proof}

We now prove Corollary~\ref{MainCor} using the main result from \cite{ciafusmag07}, which we recall here:
\begin{theorem}[Cianchi, Fusco, Maggi, Pratelli, \cite{ciafusmag07}] \label{CFMP Theorem}
There exists $C$ such that 
\begin{equation}\label{CFMP} \lambda(u)^{\zeta'}\| u \|_{L^{p^*}} \leq C( \|\na u \|_{L^p} - \SSS \| u \|_{L^{p^*}}),\end{equation}
where $\lambda(u) = \inf \big\{ \| u - v\|_{L^{p^*}}^{p^*}/\|u\|_{L^{p^*}}^{p^*} : v \in \M, \ \int |v|^{p^*} = \int |u|^{p^*} \big\} $ and $\zeta' = p^*\left(3 + 4p -\frac{3p+1}n\right)^2$.
\end{theorem} 
\begin{proof}[Proof of Corollary~\ref{MainCor}] 
As before, if \eqref{assumption2} does not hold, then Corollary~\ref{MainCor} holds trivially by simply choosing the constant to be sufficiently large. Now suppose $u\in \dot W^{1,p}$ satisfies \eqref{assumption2}. There are two obstructions to an immediate application of Theorem~\ref{CFMP Theorem}. The first is the fact that the deficit in \eqref{CFMP} is defined as $\|\na u \|_{L^p} - \SSS \| u \|_{L^{p^*}}$, while in our setting it is defined as $\|\na u \|_{L^p}^p - \SSp \| u \|_{L^{p^*}}^p$. However, this is easy to fix.
Indeed, 
using the elementary inequality
$$a^p -b^p \geq  a-b\qquad \forall\ a \geq b \geq 1,$$
we let $a = \|\na u \|_{L^p} / \SSS\|u \|_{L^{p^*}}$ and $b=1$ to get 
$$\frac{ \| \na u\|_{L^p} -\SSS \| u \|_{L^{p^*}}}{\SSS\|u\|_{L^{p^*} }}
 \leq \frac{ \| \na u\|_{L^p}^p -\SSp \| u \|_{L^{p^*}}^{p}}{\SSp\|u\|_{L^{p^*}}^p}
\leq\frac{1}{1-\delta_0}\frac{ \| \na u\|_{L^p}^p -\SSp \| u \|_{L^{p^*}}^{p}}{\|\na u\|_{L^{p}}^p},$$
where the last inequality follows from \eqref{assumption2}.
Therefore, up to increasing the constant, \eqref{CFMP} implies that
\begin{equation}\label{holds} \lambda(u)^{\zeta'} \leq C \frac{\delta(u)}{\|\na u\|_{L^p}^p}.\end{equation}

The second obstruction to applying Theorem~\ref{CFMP Theorem} is the fact that \eqref{CFMP} holds for the {\it{infimum}} in $\lambda(u)$, while we must control $\|u - v\|_{L^{p^*}}$ for $v$ attaining the infimum in \eqref{distance}. To solve this issue it is sufficient to show that there exists some constant $C=C(n,p)$ such that 
$$\int |\bar v - u|^{p^*} \leq C\inf \Big\{ \| u - v\|_{L^{p^*}}^{p^*} : v \in \M, \ \int |v|^{p^*} = \int |u|^{p^*} \Big\} $$
where $\bar v$ attains the infimum in \eqref{distance}. The proof of this fact is nearly identical (with the obvious adaptations) to that of part (2) of Proposition~\ref{small delta}, with the only nontrivial difference being that one must integrate by parts to show that the analogue of first term in \eqref{contra1} goes to zero.

Therefore, \eqref{CFMP} implies
$$\left(\frac{\| u - v\|_{L^{p^*}}}{\|u\|_{L^{p^*}}}\right)^{\zeta'} \leq C \frac{\delta(u)}{\| \na u \|_{L^p}}$$
where $v\in \M$ attains the infimum in \eqref{distance}.
Paired with Theorem~\ref{MainThm}, this proves Corollary~\ref{MainCor} with $\zeta = \zeta'p$.
\end{proof}

\section{Spectral Properties of $\LL_v$}\label{Spectral}
In this section, we give the proofs of the compact embedding theorem and Sturm-Liouville theory that were postponed in the proof of Proposition~\ref{lem: spectral properties}. As in Proposition~\ref{lem: spectral properties}, by scaling, it suffices to consider the operator $\LL = \LL_v$ where $v = v_{0,1}.$
\subsection{The discrete spectrum of $\LL$} 
\label{sect:discrete}
Given two measurable functions $\om_0, \om_1 : \Om \to \R$, let
$$W^{1,2}(\Om, \om_0,\om_1) := \{ g : \| g\|_{W^{1,2}(\Om, \om_0,\om_1)} <\infty\},$$
where $\| \cdot\|_{W^{1,2}(\Om, \om_0,\om_1)}$ is the norm defined by 
\begin{equation} \label{X norm} \| g \|_{W^{1,2}(\Om, \om_0, \om_1)} = \left( \int_\Om g^2 \om_0 + \int_\Om |\na g|^2 \om_1 \right)^{1/2}.
\end{equation}
The space $W^{1,2}_0(\Om, \om_0, \om_1)$ is defined as the completion of the space $C_0^{\infty}(\Om)$ with respect to the norm  $\| \cdot\|_{W^{1,2}(\Om, \om_0,\om_1)}$.
The following compact embedding result was shown in \cite{Opic1988}:
\begin{theorem}[Opic, \cite{Opic1988}] \label{Opic}
Let $Z = W_0^{1,2}(\Rn, \om_0, \om_1)$ and suppose 
\begin{equation}\label{intcond}\om_i \in L^1_{\rm{loc}} \quad \text{ and } \quad\om_i^{-1/2} \in  L^{2^*}_{\rm{loc}},\end{equation}
 $i=0,1$. If there are local compact embeddings
\begin{equation}\label{locemb}
W^{1,2}(B_k, \om_0, \om_1)\subset \subset L^2(B_k, \om_0), \ k \in \mathbb{N},
\end{equation}
where $B_k=\{x : |x| < k\}$, and if
\begin{equation}\label{tozero} \underset{k \to \infty}{\lim} \sup\left\{ \|u \|_{L^2(\Rn\bs B_k,\om_0)}  :  u \in Z, \ \|u \|_{Z} \leq 1 \right\} = 0,
\end{equation}
then  $Z$ embeds compactly in $ L^2(\Rn, \om_0)$.
\end{theorem}
 We apply Theorem~\ref{Opic} to show that the space 
 \begin{equation}\label{X} X = W^{1,2}_0 (\Rn, v^{p^*-2} ,  |\na v|^{p-2}),\end{equation}
 embeds compactly into $ L^2(\Rn, v^{p^*-2})$.
\begin{corollary}\label{Compact embedding}
The compact embedding $X \subset \subset L^2(\Rn, v^{p^*-2})$  holds, with $X$ as in \eqref{X}.
\end{corollary}
\begin{proof}
Let us verify that Theorem~\ref{Opic} may be applied in our setting, taking 
\[
\om_0 = v^{p^*-2}, \qquad \om_1 = |\nabla v|^{p-2}.
\]
 In other words, we must show that \eqref{intcond}, \eqref{locemb} and \eqref{tozero} are satisfied. A simple computation verifies \eqref{intcond}.
To show \eqref{locemb}, we fix $\delta>0$ small (the smallness depending only on $n$ and $p$) and show the three inclusions below:
$$ W^{1,2} (B_r, \om_0, \om_1) \overset{(1)}{\subset} W^{1, 2(n+ \delta)/(n+2)} (B_r) \overset{(2)}{\subset \subset} L^2(B_r) \overset{(3)}{\subset} L^2(B_r, \om_0).$$
Since $(2n/(2+n))^* =2$, the Rellich-Kondrachov compact embedding theorem implies $(2)$, while the inclusion $(3)$ holds simply because $v^{p^*- 2} \geq c_{n,p,r}$ for $x \in B_r$. In the direction of showing $(1)$, we use this fact and H\"{o}lder's inequality to obtain 
\begin{equation}\label{A}\Big( \int_{B_r} |u|^{2(n+\delta)/(n+2)}  \Big)^{(n+2)/(n+\delta)} 
\leq|B_r|^{(2-\delta)/(n+\delta)} \int_{B_r} |u|^2    \leq C_{n,p,r}  \int_{B_r} |v|^{p^*-2}|u|^2 .
\end{equation}
Furthermore, since
$$|\na v|^{p-2} =C(1+|x|^{p'})^{-n(p-2)/{p}} |x|^{(p-2)/(p-1)} \geq c_{n,p,r} |x|^{(p-2)/{(p-1)}} 
 \ \ \text{ for } x\in B_r,
$$
H\"{o}lder's inequality implies that 
\begin{equation}\label{B}\begin{split} 
 \Big(  \int_{B_r} |\na u|^{2(n + \delta)/(n+2)} \Big)^{(n+2)/(n + \delta)} 
 & \leq \Big(\int_{B_r} |x|^{(p-2)/(p-1)} |\na u|^2  \Big)
 \Big(\int_{B_r} |x|^{ - \beta}   \Big)^{(2-\delta)/(n+\delta)}\\
 & \leq C_{n,p,r} \int_{B_r} |\na v|^{p-2} |\na u|^2,
\end{split}\end{equation}
where $\beta=\big(\frac{p-2}{p-1}\big) \big(\frac{n+\delta}{n+2}\big)\big(\frac{n+2}{2-\delta}\big)$. Then the inclusion $(1)$ follows from \eqref{A} and \eqref{B}, and 
thus \eqref{locemb} is verified.\\

 To show \eqref{tozero}, 
 let $u_k$ be a function almost attaining the supremum in \eqref{tozero}, in other words, for a fixed $\eta>0$, let $u_k$ be such that $u_k \in X,$ $\|u_k \|_{X} \leq 1$, and 
\begin{equation*}\label{almost}\sup\left\{ \|u \|_{L^2(\Rn \bs B_k, \om_0)}  :  u \in X, \ \|u \|_{X} \leq 1 \right\} \leq \|u_k \|_{L^2(\Rn \bs B_k, \om_0)}  + \eta.
\end{equation*}
By mollifying $u$ and multiplying by a smooth cutoff $\eta\in C_0^{\infty}(\Rn\bs B_k)$, we may assume without loss of generality that $u_k \in C_0^{\infty}(\Rn\bs B_k)$. 
Recalling that $v=v_1$ with $v_1$ as in \eqref{v1}, we have 
 \begin{equation}\label{aa}
  \int_{\Rn \bs B_k}  v^{p^*-2}   u_k^2 = \int_{\Rn \bs B_k} \kappa_0(1+ |x|^{p'})^{-(p^*-2)(n-p)/p  } u_k^2  
  \leq 2\kappa_0  \int_{\Rn \bs B_k} |x|^{-(p^*-2)(n-p)/(p-1)} u_k^2 
 \end{equation}
 for $k\geq 2$. We use Hardy's inequality in the form
\begin{equation}\label{hardy}\int_{\Rn} |x|^s u^2   \leq C \int_{\Rn} |x|^{s+2}|\na u|^2  \end{equation}
for $u \in C_0^{\infty}(\Rn)$ (see, for instance, \cite{zygmund2002trigonometric}). 
Applying \eqref{hardy} to the right-hand side of \eqref{aa} implies
\begin{equation}\label{using hardy} 
\int_{\Rn \bs B_k} |x|^{-(p^*-2)(n-p) /(p-1)} u_k^2 \leq  C \int_{\Rn \bs B_k}  |x|^{-(p^*-2)(n-p)/(p-1) +2 } |\na u_k|^2 \end{equation}
and
 \eqref{aa} and \eqref{using hardy} combined give
 \begin{align*} 
   \int_{\Rn \bs B_k}  v^{p^*-2} u_k^2
 & \leq C  \int_{\Rn \bs B_k} |x|^{-(p^*-2)(n-p)/(p-1) +2 } |\na u_k|^2  \\
 & = C  \int_{\Rn \bs B_k} |x|^{-p'} |x|^{-(p-2)(n-1)/(p-1)} |\na u_k|^2  \\
 & \leq C k^{-p'}\int_{\Rn \bs B_k} |\na v|^{p-2}   |\na u_k|^2 ,
 \end{align*}
 where the final inequality follows because
 $$ |\nabla v|^{p-2} \geq C | x|^{-(p-2)(n-1)/(p-1)}   \ \ \text{ for } x \in \Rn\bs B_1.$$
Thus 
$$
 \int_{\Rn \bs B_k}  v^{p^*-2}    u_k^2\leq Ck^{-p'} \| u_k \|_X,$$
 and \eqref{tozero} is proved.
\end{proof}
Thanks to the compact embedding $X \subset \subset L^2(\Rn, \om_0)$, we can now prove the following important fact:

\begin{corollary}\label{discrete}
The operator $\LL$ has a discrete spectrum $\{\alpha_i\}_{i=1}^{\infty}$. \end{corollary}
\begin{proof} 
We show that the operator $\LL^{-1} : L^2(v^{p^*-2} ) \to L^2(v^{p^*-2})$ is bounded, compact, and self-adjoint. From there, one applies the spectral theorem (see for instance \cite{evansPDE}) to deduce that $\LL^{-1}$
has a discrete spectrum, hence so does $\LL$.

 Approximating by functions in $ C_0^{\infty}(\Rn),$ the Poincar\'{e} inequality \eqref{Poincare} holds for all functions $\vphi 
 \in X$, with $X$ as defined in \eqref{X}. Thanks to this fact,
  the existence and uniqueness of solutions to $\LL u = f$ for $f \in L^2(v^{p^*-2})$ follow from the Direct Method, so the operator $\LL^{-1}$ is well defined.
  
 Self-adjointness is immediate.
From \eqref{Poincare} and H\"{o}lder's inequality, we have 
\begin{align*}
c \| u \|_{X}^2 \leq \int |\na v|^{p-2} |\na u|^2 \leq \int A[\na u, \na u] 
\leq \| u\|_{X} \|\LL u\|_{L^2(v^{p^*-2})} .
\end{align*}
This proves that $\LL^{-1}$ is bounded from $ L^2(v^{p^*-2} ) $ to $L^2(v^{p^*-2})$, and by Corollary~\ref{Compact embedding} we see that $\LL^{-1}$ is a compact operator.
\end{proof}

\subsection{Sturm-Liouville theory}
Multiplying by the integrating factor $r^{n-1}$, the ordinary differential equation \eqref{Req} takes the form of the Sturm-Liouville eigenvalue problem 
\begin{equation} \label{SL equation}
Lf + \alpha f =0 \ \ \text {on } \ \ [0,\infty),\end{equation} where 
$$Lf= \frac{1}{w}[(Pf')' - Qf]$$ with
\begin{equation}\label{PQw}\begin{split}
P(r) &= (p-1) |v'|^{p-2} r^{n-1}, \\
Q(r)& = \mu r^{n-3} |v'|^{p-2},\\
w(r) & = v^{p^*-2}r^{n-1}.
\end{split}
\end{equation}

This is a \textit{singular} Sturm-Liouville problem; first of all, our domain is unbounded, and second of all, the equation is degenerate because $v'(0)=0$. Nonetheless, we show that Sturm-Liouville theory holds for this singular problem.

\begin{lemma}[Sturm-Liouville Theory]\label{Sturm-Liouville}
The following properties hold for the singular Sturm-Liouville eigenvalue problem \eqref{SL equation}:
\begin{enumerate}
\item If $f_1$ and $f_2$ are two eigenfunctions corresponding to the eigenvalue $\alpha$, then $f_1 = c f_2$. In other words, each eigenspace of $L$ is one-dimensional.
\item The $i$th eigenfunction of $L$ has $i-1$ interior zeros.
\end{enumerate}
\end{lemma}

Note that $L$ has a discrete spectrum because $\LL$ does (Corollary~\ref{discrete}), 
and that eigenfunctions $f$ of $L$ live in the space 
$$Y = W^{1,2}_0\big([0, \infty), v^{p^*-2} r^{n-1}, |v'|^{p-2}r^{n-1}\big),$$
using the notation introduced at the beginning of Section~\ref{sect:discrete}.
In any ball $B_R$ around zero, the operator $\LL$ is degenerate elliptic with the matrix $A$ bounded by an $A_2$-Muckenhoupt weight, so eigenfunctions of $\LL$ are H\"{o}lder continuous; see \cite{Fabes82, gutierrez1989harnack}.
Therefore, eigenfunctions of $L$ are H\"{o}lder continuous on $[0, \infty)$.

\begin{remark}\label{integral converges}{\rm{The function $P(r)$ as defined in \eqref{PQw} has the following behavior:
\begin{align*} P(r) &\approx r^{(p-2)(p-1) +n-1}\quad   \text{ in } [0,1],\\
P(r)& \approx r^{(n-1)/(p-1)} \quad \text{ as } r \to \infty.
\end{align*}
In particular, the weight $|v'|^{p-2}r^{n-1}\approx r^{(n-1)/(p-1)}$ goes to infinity as $r \to \infty$, which implies that $\int_1^{\infty} |f'|^2 dr <\infty$
for any $f\in Y$.
}}
\end{remark}

In order to prove Lemma~\ref{Sturm-Liouville}, we first prove the following lemma, which describes the asymptotic decay of solutions of \eqref{SL equation}.

\begin{lemma}\label{Reg and Asym}
Suppose $f \in Y$ is a solution of \eqref{SL equation}. Then, for any $0<\beta< \frac{n-p}{p-1}$, there exist $C$ and $r_0$ such that 
$$|f(r)| \leq Cr^{-\beta } \ \ \text{ and }  \ \ |f'(r) | \leq Cr^{-\beta -1}$$
for $r \geq r_0$.
\end{lemma}
\begin{proof}

\noindent{\it{Step 1: Qualitative Decay of $f$.}}
For any function $f\in Y$, $f(r ) \to 0$ as $r \to \infty.$ Indeed, near infinity, $|v'|^{p-2 }r^{p-1}$ behaves like $C r^{\g}$ where $\g := \frac{n-1}{p-1}>1$. 
Then for any $r,s$ large enough with $r<s$, 
\begin{equation}
\label{eq:regularity f}
|f(r) - f(s) |  \leq   \int_r^\infty| f'(t)| dt   \leq \Big( \int_r^{\infty} f'(t)^2 t^{\g} dt \Big)^{1/2} \Big( \int_r^{\infty} t^{-\g} dt \Big)^{1/2}
\end{equation}
by H\"{o}lder's inequality. As both integrals on the right-hand side of \eqref{eq:regularity f} converge, for any $\e>0$, we may take $r$ large enough such that the right-hand side is bounded by $\e$, so the limit of $f(r)$ as $r \to \infty$ exists.

We claim that this limit must be equal to zero.
Indeed, since $Y$ is obtained as a completion of $C_0^\infty$, if we apply \eqref{eq:regularity f} to a sequence $f_k \in C_0^{\infty}([0,\infty))$ converging in $Y$ to
$f$ and we let $s \to \infty$, we get
$$ |f_k(r) | \leq 
 \Big( \int_r^{\infty} f_k'(t)^2 t^{\g} dt \Big)^{1/2} \Big( \int_r^{\infty} t^{-\g}dt \Big)^{1/2},$$
thus, by letting $k\to \infty$,
$$ |f(r) | \leq \Big( \int_r^{\infty} f'(t)^2 t^{\g} dt \Big)^{1/2} \Big( \int_r^{\infty} t^{-\g}dt \Big)^{1/2}.$$
Since the right-hand side tends to zero as $r\to\infty$, this proves the claim.\\

\noindent{\it{Step 2: Qualitative Decay of $f'$.}}
For $r> 0$, \eqref{SL equation} can be written as
\begin{equation}\label{Dividing Through} L'f :=  f''  + a f' + b f = 0
\end{equation}
where $$a = \frac{P'}{P} \ \ \text{ and } \ \ b=\frac{-Q + w\alpha}{P}.$$
Fixing $\e >0$, an explicit computation shows that there exists $r_0$ large enough such that
$$ \frac{(1-\e)(n-1)}{p-1}\, \frac{1}{r}  \leq a \leq \frac{(1+\e) (n-1)}{p-1}\, \frac{1}{r}$$
and 
$$-\frac{\mu}{p-1}\frac{1}{r^2}+\frac{ (1-\e)  c_{p,n}\alpha}{ r^{(3p-2)/(p-1)}}  \leq b \leq 
- \frac{\mu}{p-1}\frac{1}{r^2} + \frac{(1+\e)  c_{p,n}\alpha}{ r^{(3p-2)/(p-1)}} $$
for $r\geq r_0$, where $c_{n,p}$ is a positive constant depending only on $n$ and $p$.
Asymptotically, therefore, our equation behaves like
$$f'' + \frac{n-1}{p-1} \frac{f'}{r}+ \Big(\frac{  c_{p,n}\alpha}{ r^{p'}} - \frac{\mu}{p-1}\Big) \frac{f}{r^2}=0.$$
\\
If $f$ is a solution of \eqref{SL equation}, then squaring \eqref{Dividing Through} on $[r_0, \infty)$, we obtain
$$|f ''|^2  \leq 2\left(\Big( \frac{n-1}{p-1}+ \e\Big)\frac{f'}{r}\right)^2 + 2\left(\Big(\frac{ (1+\e) c_{p,n}\alpha }{ r^{p'}}+ \frac{\mu}{p-1} \Big) \frac{f}{r^2} \right)^2 \leq C(|f|^2  + |f'|^2).$$
Integrating on $[R, R+1]$ for $R \geq r_0$ implies
$$\int_{R}^{R+1} |f''|^2 \leq C \int_R^{R+1} |f'|^2 + C \int_R^{R+1} |f|^2.$$ 
Step $1$ and Remark~\ref{integral converges} ensure that both terms on the right-hand side go to zero.
Applying Morrey's embedding to $f'\eta_R$, where $\eta_R$ is a smooth cutoff equal to $1$ in $[R,R+1]$, we determine that
$\| f'\|_{L^{\infty}([R, R+1])} \to 0$ as $R\to \infty$, proving that $f'(r) \to 0$ as $r\to \infty.$ 
\\

\noindent{\it{Step 3: Quantitative Decay of $f$ and $f'$.}}
Standard arguments (see for instance \cite[VI.6]{courant2008methods}) show that, also in our case,
the $i$th eigenfunction $f$ of $L$ has at most $i-1$ interior zeros; in particular, $f(r)$ does not change sign for $r$ sufficiently large. Without loss of generality, we assume that eventually $f\geq 0$. 

Taking $r_0$ as in Step $2$ and applying the operator $L'$ defined in \eqref{Dividing Through} to the function $g =Cr^{-\beta}+c$, $c>0$, for $r\geq r_0$ gives
\begin{align*}
L'g &\leq C \beta(\beta+1) r^{-\beta -2} - \frac{(1-\e)(n-1)}{p-1} C\beta r^{-\beta - 2} +\Big(\frac{(1+\e) c_{p,n} \alpha}{ r^{(3p-2)/(p-1)}}- \frac{\mu}{p-1}\Big)(C r^{-\beta - 2 }+c)\\
& \leq C r^{-\beta -2} \Big(\beta(\beta+1)- \frac{(1-\e)(n-1)}{p-1}\beta + \frac{(1+\e)  c_{p,n}\alpha}{ r^{p'}}\Big) + \frac{(1+\e)  c_{p,n}\alpha}{ r^{(3p-2)/(p-1)}}c.
\end{align*}
For any $0<\beta <(n-p)/(p-1)$, $r_0$ may be taken large enough (and therefore $\e$ small enough) such that
$$L'g<0 \quad \text{ on }\quad[r_0, \infty),$$ so $g$ is a supersolution of the equation on this interval.

Choosing  $C= f(r_0) r_0^{\beta}$ and $c>0$, then $(g-f)(r_0) > 0$ and $(g-f)(r) \to c>0$ as $r \to \infty$.
Since $L'(g-f) < 0,$ we claim that $g-f>0$ on $(r_0, \infty)$. Indeed, otherwise, $g-f$ would have a negative minimum at some $r \in (r_0, \infty)$, implying that
$$(g-f)(r) \leq 0, \ \ (g-f)'(r)=0, \ \ \text{ and } \ \ (g-f)''(r) \geq 0,$$ forcing  $L'(g-f) \geq 0,$ a contradiction.
This proves that $0 \leq f \leq g$ on $[r_0,\infty)$, and since $c>0$ was arbitrary, we determine that
$f \leq C r^{-\beta}$
on $[r_0,\infty)$.\\

We now derive bounds on $f'$:
by the fundamental theorem of calculus and using \eqref{Dividing Through} and the bound on $f$ for $r\geq r_0$, we get
\begin{align*}
|f'(r) | &= \Big| \int_r^{\infty} f''\Big| \leq \frac{C}{r} \left|\int_r^{\infty} f'\right| +C \left|\int_r^{\infty} t^{-\beta -2} \right| \leq \frac{C}{r} |f(r) | +  \frac{C}{\beta+2} r^{-\beta -1} \leq C  r^{-\beta -1}.
\end{align*}
\end{proof}

With these asymptotic decay estimates in hand, we are ready to prove Lemma~\ref{Sturm-Liouville}.
\begin{proof}[Proof of Lemma~\ref{Sturm-Liouville}]
We begin with the following remark about uniqueness of solutions.
If $f_1$ and $f_2$ are two solutions of \eqref{SL equation} 
and  $$f_1(r_0) = f_2(r_0), \quad f_1'(r_0)  =f_2'(r_0)$$ for some $r_0 >0$, then $f_1=f_2$ on $[0,\infty)$.
Indeed, for $r > 0$, we may express our equation as in \eqref{Dividing Through}.
As $a$ and $b$ are continuous on $(0,\infty),$ the standard proof of uniqueness for (non-degenerate) second order ODE holds. 
Once $f_1=f_2$ on $(0,\infty)$, they are also equal at $r=0$ by continuity.\\

\noindent {\it{Proof of (1).}} Suppose $\alpha$ is an eigenvalue of $L$ with $f_1$ and $f_2$ satisfying \eqref{SL equation}.
In view of the uniqueness remark, if there exists $r_0 >0$ and some linear combination $f$ of $f_1$ and $f_2$ such that $f(r_0) = f'(r_0 ) = 0,$ then $f$ is constantly zero and $f_1$ and $f_2$ are linearly dependent.
Let
 $$W(r) = W(f_1, f_2) (r)  := \det \left[ \begin{array}{cc} f_1 & f_2 \\ f_1' &  f_2' \end{array} \right] (r)$$
denote the Wronskian of $f_1$ and $f_2$.  This is well defined for $r>0$
(since $f_1$ and $f_2$ are $C^2$ there) and a standard computation shows that $ (PW)'=0$ on $(0,\infty)$:  
indeed, since
$W' = f_1f_2''  - f_2f_1''$, we get
$$(PW)' = PW' + P'W= P( f_1f_2''  - f_2f_1'') + P'(f_1 f_2 ' - f_2f_1'),$$
and by adding and subtracting the term $(\alpha w - Q)f_1f_2$ it follows that
\begin{align*}(PW)'
& =f_1 \left(P f_2'' + P' f_2'  + (\alpha w - Q)f_2 \right) - f_2\left( P f_1'' + P'f_2' + (\alpha w -Q) f_1\right) = 0.
\end{align*}
Thus $PW$ is constant on $(0,\infty)$. 
We now show that that $PW$ is continuous up to $r=0$ and that $(PW)(0) = 0.$
Indeed, \eqref{SL equation} implies that
$$(Pf'_i)' = (Q-\alpha w)f_i$$ for $i=1,2$. 
The right-hand side is continuous, so $(Pf_i')'$ is continuous, from which it follows easily that $PW$ is also continuous on $[0,\infty)$.

To show  that $(PW)(0)=0$, we first prove that $(Pf_i')(0)=0$. Indeed, let
$ c_i:=(Pf_i')(0)$. If $c_i\neq 0$, then keeping in mind Remark~\ref{integral converges}, 
\begin{equation}\label{Limit Value}  f_i'(r) \approx \frac{c_i}{P(r)}\approx \frac{c_i}{r^{(p-2)/(p-1) + n -1}} \qquad \text{for $r \ll 1$},\end{equation}
therefore
 \begin{align*}
 \int_0^R |v'|^{p-2} |f'|^2 r^{n-1} dr \gtrsim \int_0^{R} r^{(p-2)/(p-1)+n-1} |f'|^2 dr \gtrsim \int_0^R  \frac{dr}{r^{(p-2)/(p-1)+n-1} }=+\infty,
 \end{align*}
contradicting the fact that $f \in Y$. Hence, we conclude that $\underset{r\to 0}{\lim} (Pf_i')(r) = 0,$
and using this fact we obtain
$$(PW)(0) = \underset{r \to 0}{\lim}\,( Pf_1' f_2 - Pf_2'f_1)
= \underset{r \to 0}{\lim}\, ( Pf_1' )\ \underset{r \to 0}{\lim}\, f_2
- \underset{r \to 0}{\lim}\, ( P f_2')\  \underset{r \to 0}{\lim}\, f_1=0.$$
Therefore $(PW)(r)=0$ for all $r\in [0,\infty)$. Since $P(r) >0 $ for $r> 0$, we determine that $W(r) = 0 $ for all $r > 0$. In particular, given $r_0\in (0,\infty)$,
there exist $c_1$, $c_2 $ such that $c_1^2 + c_2^2 \neq 0$ and 
\begin{align*}
c_1f_1(r_0 ) + c_2 f_2 (r_0) & = 0,\\
c_1f_1'(r_0 ) + c_2 f_2' (r_0) & = 0.
\end{align*}
Then $f := c_1f_1 + c_2 f_2$ solves \eqref{SL equation} and $f(r_0)  =f'(r_0) = 0$. By uniqueness, $f\equiv 0$ for all $t \in (0,\infty)$, and so $f_1=c f_2$.\\

\noindent {\it{Proof of (2).}} \  Thanks to our preliminary estimates on the behavior of $f_i$ at infinity, the following is an adaptation of the standard argument 
in, for example, \cite[VI.6]{courant2008methods}.

Suppose  that $f_1$ and $f_2$ are eigenfunctions of $L$ corresponding to eigenvalues $\alpha_1$ and $\alpha_2$ respectively, with $\alpha_1 < \alpha_2$, that is,
$$ (Pf_i')' - Qf_i + \alpha_i w f_i = 0.$$
Our first claim is that between any two consecutive zeros of $f_1$ is a zero of $f_2$,  including zeros at infinity.  
Note that 
\begin{equation}\label{PW prime} \begin{split}
(PW)'  & = P[f_1f_2''-  f_2 f_1''] + P'[f_1f_2' - f_2 f_1'] \\
& = f_1[ (Pf_2')'  + (\alpha_2 -Q) f_2] -
f_2 [ (Pf_1')' + (\alpha_1w - Q) f_1] + (\alpha_1 - \alpha_2 )w f_1 f_2\\
& =(\alpha_1 - \alpha_2 )w f_1 f_2.
\end{split}\end{equation}

Suppose that $f_1$ has consecutive zeros at $r_1$ and $r_2$, and suppose for the sake of contradiction that $f_2$ has no zeros in the interval $(r_1, r_2)$. With no loss of generality, we may assume that $f_1$ and $f_2 $ are both nonnegative in $[r_1, r_2]$. \\
\noindent \textit{Case 1:} Suppose that $r_2<\infty$. Then integrating \eqref{PW prime} from $r_1$ to $r_2$ implies
\begin{align*} 0&>  (\alpha_1 - \alpha_2 ) \int_{r_1}^{r_2} w f_1 f_2 = (PW)(r_2) -(PW)(r_1) \\
&  = P(r_2) [f_1(r_2) f_2'(r_2) - f_1'(r_2) f_2(r_2) ] -  P(r_1) [f_1(r_1) f_2'(r_1) - f_1'(r_1) f_2(r_1) ]\\
& = -P(r_2) f_1'(r_2) f_2(r_2) + P(r_1) f_1'(r_1) f_2(r_1).
\end{align*}
The function $f_1$ is positive on $(r_1, r_2)$, so $f_1'(r_1) \geq 0$ and $f_1'(r_2) \leq 0$.
Also, since $f_1(r_1)=f_1(r_2)=0$ we cannot have $f_1'(r_1) = 0$ or $f_1'(r_2)=0$, as otherwise $f_1$ would vanish identically.
Furthermore, $f_2$ is nonnegative on $[r_1, r_2]$, so we conclude that the right-hand side is nonnegative, giving us a contradiction.\\
\noindent \textit{Case 2:} Suppose that $r_2 = \infty$. Again integrating the identity \eqref{PW prime} from $r_1$ to $\infty$, we obtain
\begin{equation}\label{Zero at infinity}\begin{split} 
 0&>  (\alpha_1 - \alpha_2 ) \int_{r_1}^{\infty} w f_1 f_2 =  \underset{r\to \infty}{\lim}(PW)(r) - (PW)(r_1)\\
&  = \underset{r\to \infty}{\lim}[P(r) (f_1(r) f_2'(r) - f_1'(r) f_2(r) )] -  P(r_1) (f_1(r_1) f_2'(r_1) - f_1'(r_1) f_2(r_1) ).
\end{split}\end{equation}
We notice that Lemma~\ref{Reg and Asym} implies that
$$  \underset{r\to \infty}{\lim}[P(r) (f_1(r) f_2'(r) - f_1'(r) f_2(r) )]  =0.$$
Indeed, taking $\frac{n-p}{2(p-1)}<\beta < \frac{n-p}{p-1}$,
$$| f_1' f_2 - f_1f_2'| \leq |f_1'||f_2| + |f_1| |f_2'| \leq C r^{-2\beta-1},$$
and, recalling Remark~\ref{integral converges},
$$P(r) \leq C r^{(n-1)/(p-1)},$$
implying that
$$P\,|f_1' f_2 - f_1f_2'| \leq C r^{\g}\to 0,$$
where $\g =- 2\beta - 1 + \frac{n-1}{p-1} < 0.$
Then \eqref{Zero at infinity} becomes
$$0 >
-P(r_1) f_1'(r_1) f_2(r_1).$$ 
Since $f_1'(r_1) >0$ and $f_2(r_1)\geq 0$ (see the argument in Case $1$), this gives us a contradiction.

We now claim that $f_2$ has a zero in the interval $[0, r_1)$, where $r_1$ is the first zero of $f_1$. Again, we assume for the sake of contradiction that $f_2$ has no zero in this interval and that, without loss of generality, $f_1$ and $f_2$ are nonnegative in $[0,r_1]$.
Integrating \eqref{PW prime} implies
\begin{equation}\label{Zero near zero} 0>  (\alpha_1 - \alpha_2 ) \int_{0}^{r_1} w f_1 f_2 = PW(r_1) - PW(0).
\end{equation}
The same computation as in the proof of Part $(1)$ of this lemma implies that $(PW)(0) = 0$,
so \eqref{Zero near zero} becomes
$$0> -P(r_1)  f_1'(r_1) f_2(r_1), $$
once more
giving us a contradiction.\\

The first eigenfunction of an operator is always positive in the interior of the domain, so the second eigenfunction of $L$ must have at least one interior zero by orthogonality. Thus the claims above imply that the $i$th eigenfunction has at least $i-1$ interior zeros. On the other hand, as mentioned in the proof of Lemma~\ref{Reg and Asym}, the standard theory also implies that  the $i$th eigenfunction has at most $i-1$ interior zeros, and the proof is complete.
\end{proof}

\section{Appendix}\label{appendix}
In this section we give the proofs of Lemma~\ref{numbers} and of the polar coordinates form of the operator $\DIVV(A(x) \na \vphi)$ given in \eqref{polarA}.

 \begin{proof}[Proof of \eqref{polarA}]
We will use the following classical relations:
\begin{align*}
\partial_{r} \hat{r} = 0& \qquad \partial_{r} \hat{\theta}_i = 0, \qquad \partial_{\theta_i}\hat{r} = \hat{\theta}_i , \qquad
\partial_{\theta_i}\hat{\theta}_i = -\hat{r}, \qquad
\partial_{\theta_j}\hat{\theta}_i =0\quad \text{for }i \neq j.
\end{align*} 
The chain rule implies that
\begin{align}\label{operator}
\text{div} (A(x) \na \varphi) & = \text{tr}( A(x) \na^2 \varphi ) + \text{tr}(\na A(x) \na \varphi).
\end{align}
We compute the two terms on the right-hand side of \eqref{operator} separately.
For the first, we begin by computing the Hessian of $\varphi$ in polar coordinates, starting from
\begin{equation}\label{gradient polar}
\na \vphi = \partial_r \varphi \, \hat{r} + \frac{1}{r}\sum_{j=1}^{n-1}  \partial_{\theta_j} \varphi \,\hat{\theta}_j,
\end{equation}
We have
\begin{align*}
\na^2 \varphi &
= \partial_r \Big( \partial_r \varphi \, \hat{r} + \frac{1}{r}\sum_{j=1}^{n-1}  \partial_{\theta_j} \varphi \,\hat{\theta}_j\Big) \hat{r}
+ \frac{1}{r} \sum_{i=1}^{n-1}\partial_{\theta_i} 
\Big( \partial_r \varphi\, \hat{r} + \frac{1}{r}\sum_{j=1}^{n-1} \partial_{\theta_j} \varphi \, \hat{\theta}_j\Big)\hat{\theta}_i\\ 
& 
= \partial_{rr} \varphi \, \hat{r} \otimes \hat{r}
-\frac{1}{r^2}\sum_{j=1}^{n-1}\partial_{\theta_j}\varphi\, \hat{\theta}_j\otimes \hat{r}
+\frac{1}{r} \sum_{j=1}^{n-1}\partial_{\theta_j r} \varphi\, \hat{\theta}_j\otimes \hat{r} \\
& +\frac{1}{r} \sum_{i=1}^{n-1}  \partial_{\theta_i r} \varphi \,\hat{r} \otimes \hat{\theta}_i
+\frac{1}{r}\sum_{i=1}^{n-1} \partial_r \varphi \,\theta_i \otimes \theta_i 
+ \frac{1}{r^2} \sum_{i=1}^{n-1} \sum_{j=1}^{n-1} \partial_{\theta_i \theta_j}
\varphi\, \hat{\theta}_j\otimes \hat{\theta}_i
- \frac{1}{r^2}\sum_{i=1}^{n-1}
 \partial_{\theta_i} \varphi\,  \hat{r} \otimes \hat{\theta}_i.
\end{align*}
In order to compute $A(x) \na^2\varphi$, we note that
\begin{equation*} 
(\hat{r} \otimes \hat{r} )(\hat{r} \otimes \hat{r} ) =\hat{r} \otimes \hat{r},  
\qquad (\hat{r} \otimes \hat{r} )(\hat{\theta}_j \otimes \hat{\theta}_i ) = 0, \qquad (\hat{r} \otimes \hat{r} )(\hat{r} \otimes\hat{\theta}_i)  = 0,
\qquad
 (\hat{r} \otimes \hat{r} )(\hat{\theta}_i \otimes \hat{r} )  =\hat{\theta}_i \otimes \hat{r}.
 \end{equation*}
Thus we have
\begin{align*}
A(x) \na^2 \varphi & = (p-2)|\na v|^{p-2} \hat{r}\otimes \hat{r} (\na^2 \varphi)+ |\na v|^{p-2} \text{Id}(\na^2 \varphi)\\
&=
(p-2) |\na v|^{p-2} \Big[ \partial_{rr} \varphi\, \hat{r} \otimes \hat{r} - \frac{1}{r^2} \sum_{j=1}^{n-1}\partial_{\theta_j} \varphi\, \hat{\theta}_j\otimes \hat{r} 
+ \frac{1}{r}\sum_{j=1}^{n-1} \partial_{\theta_j r } \varphi \, \hat{\theta}_j \otimes \hat{r} \Big]\\
& + |\na v|^{p-2}\Big[\partial_{rr} \varphi\, \hat{r} \otimes \hat{r}
-\frac{1}{r^2} \sum_{j=1}^{n-1}\partial_{\theta_j}\varphi\, \hat{\theta}_j\otimes \hat{r}
+\frac{1}{r}\sum_{j=1}^{n-1} \partial_{\theta_j r} \varphi\, \hat{\theta}_j\otimes \hat{r} \\
& + \frac{1}{r}\sum_{i=1}^{n-1} \partial_{\theta_i r} \varphi \,\hat{r} \otimes \hat{\theta}_i
+\frac{1}{r}\sum_{i=1}^{n-1}\partial_r \varphi \,\theta_i \otimes \theta_i 
+ \frac{1}{r^2}\sum_{i=1}^{n-1}\sum_{j=1}^{n-1} \partial_{\theta_i \theta_j} \varphi \,\hat{\theta}_j\otimes \hat{\theta}_i
- \frac{1}{r^2} \sum_{i=1}^{n-1}\partial_{\theta_i} \varphi\, \hat{r} \otimes \hat{\theta}_i\Big],
\end{align*}
and the first term in \eqref{operator} is
\begin{equation}\label{term 1} \begin{split} 
 \text{tr}( A(x) \na^2 \varphi ) 
 & = (p-1) |\na v|^{p-2} \partial_{rr} \varphi +  \frac{n-1}{r} |\na v|^{p-2}\partial_r \varphi +\frac{1}{r^2}|\na v|^{p-2}\sum_{i=1}^{n-1} \partial_{\theta_i \theta_i } \varphi.
 \end{split}\end{equation}
Now we compute the second term in \eqref{operator}, starting by computing $\na A(x)$. We reintroduce the slight abuse of notation by letting $v(r) = v(x)$, so $v' = \pa_r v$, $v'' = \pa_{rr} v$. 
Note that $\partial_{\theta} \text{Id} =\partial_r \text{Id} = 0$, thus
\begin{align*}
\na A(x)  &= \partial_r A(x) \otimes \hat{r}
 + \frac{1}{r}\sum_{j=1}^{n-1} \partial_{\theta_j} A(x) \otimes \hat{\theta}_j \\
 &= (p-2)^2 |v'|^{p-4}  v' \,  v'' \,\hat{r} \otimes \hat{r}\otimes \hat{r}  +  (p-2)|v'|^{p-4} v' \, v''\, \text{Id}\otimes \hat{r}\\
 & + \frac{p-2}{r} \sum_{j=1}^{n-1}\Big[ |v'|^{p-2} \hat{\theta}_j\otimes \hat{r} \otimes \hat{\theta}_j + |v'|^{p-2} \hat{r} \otimes \hat{\theta}_j \otimes \hat{\theta}_j\Big].
\end{align*}
Recalling \eqref{gradient polar}, we then have
\begin{align*}
\na A(x)
 \na \varphi 
 & = 
(p-2)^2 |v'|^{p-4}v'\, v''\, \partial_r \varphi
(\hat{r} \otimes \hat{r} \otimes \hat{r} )\hat{r} 
+ (p-2) |v'|^{p-4}  v'\,  v''\, \partial_r \varphi (\text{Id} \otimes \hat{r} )\hat{r}\\
& + \frac{p-2}{r} \sum_{j=1}^{n-1}\left[  |v'|^{p-2} \partial_r \varphi
 (\hat{\theta}_j\otimes \hat{r}\otimes \hat{\theta}_j ) \hat{r}
+ |v'|^{p-2}\partial_r \varphi (\hat{r} \otimes \hat{\theta}_j \otimes \hat{\theta}_j) \hat{r}\right] \\
&
+ \frac{1}{r}\sum_{i=1}^{n-1}
 \left[ (p-2)^2 |v'|^{p-4}  v' v'' \partial_{\theta_i} \varphi (\hat{r} \otimes \hat{r}\otimes\hat{r} )\hat{\theta}_i
+ (p-2) |v'|^{p-4}  v'  v'' \partial_{\theta_i} \varphi (\text{Id} \otimes \hat{r} )\hat{\theta}_i \right]\\
& +\frac{p-2}{r^2} \sum_{i=1}^{n-1} \sum_{j=1}^{n-1}
\left[  | v'|^{p-2} \partial_{\theta_i} \varphi (\hat{\theta}_j\otimes \hat{r} \otimes \hat{\theta}_j) \hat{\theta}_i
+  |v'|^{p-2} \partial_{\theta_i} \varphi (\hat{r} \otimes \hat{\theta}_j \otimes \hat{\theta}_j )\hat{\theta}_i\right],
\end{align*}
where we used that $(a\otimes b\otimes c)d = (a\cdot d) b\otimes c.$ Writing out these terms gives 
\begin{align*}
\na A(x) \na \vphi
&=
(p-1)(p-2) |v'|^{p-4}v' v'' \partial_r \varphi \,
\hat{r} \otimes \hat{r} 
+ \frac{p-2}{r} |v'|^{p-2}\sum_{j=1}^{n-1}\partial_r \varphi \, \hat{\theta}_j \otimes \hat{\theta}_j\\
&+ \frac{p-2}{r} |v'|^{p-4} v' v''\sum_{j=1}^{n-1} \partial_{\theta_j} \varphi\, \hat{\theta}_j\otimes \hat{r} +\frac{p-2}{r^2}  |v'|^{p-2}\sum_{j=1}^{n-1} \partial_{\theta_j} \varphi \, \hat{r} \otimes \hat{\theta}_j,
\end{align*}
thus the second term in \eqref{operator} is
\begin{equation}\label{term 2}
\text{tr}(\na A(x) \na \varphi)  
= (p-1)(p-2)|\na v|^{p-4} \partial_r v\, \partial_{rr} v\, \partial_r \varphi+ \frac{(n-1)(p-2)}{r} |\na v|^{p-2} \partial_{r} \varphi.
\end{equation}
Combining \eqref{term 1} and \eqref{term 2}, \eqref{operator} implies that
\begin{align*}
\text{div} (A(x) \na \varphi) &=
(p-1) |\na v|^{p-2} \partial_{rr} \varphi +  \frac{(p-1)(n-1)}{r} |\na v|^{p-2}\partial_r \varphi +\frac{1}{r^2}|\na v|^{p-2}\sum_{j=1}^{n-1} \partial_{\theta_j \theta_j } \varphi
\\ &+
(p-1)(p-2)|\na v|^{p-4} \partial_r v\, \partial_{rr} v \,\partial_r \varphi,
\end{align*}
as desired.

\end{proof}

 \bibliographystyle{plain}
\bibliography{references4}

\end{document}